\documentclass[a4paper,oneside]{amsart} 
\usepackage[english]{babel} 
\usepackage{amsmath, csquotes, amssymb, tabu, booktabs,xparse,bbm, amsthm} 
\usepackage{geometry} 
\usepackage{tikz} 
\usepackage{enumitem} 
\usepackage{stmaryrd} 
\usepackage{scalerel,stackengine}   
\usepackage{graphicx}

\usetikzlibrary{decorations.markings}

\usepackage{hyperref}

\theoremstyle{definition}
\newtheorem{defn}{Definition}[section]
\newtheorem{thm}[defn]{Theorem}
\newtheorem{lem}[defn]{Lemma}
\newtheorem{prop}[defn]{Proposition}
\newtheorem{ex}[defn]{Example} 
\newtheorem{cor}[defn]{Corollary}
\newtheorem{rmk}[defn]{Remark}

\stackMath
\newcommand\reallywidehat[1]{%
\savestack{\tmpbox}{\stretchto{%
  \scaleto{%
    \scalerel*[\widthof{\ensuremath{#1}}]{\kern-.6pt\bigwedge\kern-.6pt}%
    {\rule[-\textheight/2]{1ex}{\textheight}}
  }{\textheight}%
}{0.6ex}}%
\stackon[1pt]{#1}{\tmpbox}%
}

\stackMath
\newcommand\reallywidetilde[1]{%
\savestack{\tmpbox}{\stretchto{%
  \scaleto{%
    \scalerel*[\widthof{\ensuremath{#1}}]{\kern-.6pt\sim\kern-.6pt}%
    {\rule[-\textheight/2]{1ex}{\textheight}}
  }{\textheight}%
}{0.9ex}}%
\stackon[1pt]{#1}{\tmpbox}%
} 

\NewDocumentCommand\Rnum{}{\mathbbm{R}}

\NewDocumentCommand\Cnum{}{\mathbbm{C}}

\NewDocumentCommand\dx{}{\operatorname{d\!}}
\NewDocumentCommand\party{mmO{}}{\frac{\partial^{#3} #1}{\partial^{#3} #2}}

\NewDocumentCommand\End{}{\operatorname{End}}

\NewDocumentCommand\lb{}{\llbracket} 
\NewDocumentCommand\rb{}{\rrbracket} 

\NewDocumentCommand\abs{m}{\vert #1 \vert}

\NewDocumentCommand\Zsig{}{Z_{\Sigma}}

\NewDocumentCommand\res{}{\operatorname{res}}
\NewDocumentCommand\DTM{}{\mathbbm{T}M}

\NewDocumentCommand\Symp{}{\operatorname{Symp}}
\NewDocumentCommand\Flux{}{\operatorname{Flux}} 
 
\NewDocumentCommand\Id{}{\operatorname{Id}}

\title[Branes with boundary and symplectic methods for stable GC geometry]{Lagrangian branes with boundary and symplectic methods for stable generalized complex manifolds} 
\author{Charlotte Kirchhoff-Lukat} 

\begin{document} 
\maketitle 

\begin{abstract} 
Generalized complex (GC) geometry interpolates between ordinary symplectic and complex geometry. 
Stable generalized complex manifolds (first introduced by Cavalcanti, Gualtieri in 2015, \cite{Cavalcanti2015a}) carry a Poisson structure which is generically symplectic, but degenerates on a (real) codimension-2 submanifold. Up to gauge equivalence, the stable generalized complex structure is thus determined by what is called an elliptic symplectic form, which allows the extension of a number of techniques and results from symplectic geometry to stable GC geometry. \\
This paper introduces a new type of submanifold in stable GC manifolds: Lagrangian branes with boundary, which are generically Lagrangian and intersect the degeneracy locus in their boundary. By relating stable GC manifolds to log symplectic manifolds, we are able to prove results on local neighbourhoods and small deformations of such branes. 
We further investigate stable generalized complex Lefschetz fibrations, where Lagrangian branes with boundary arise as Lefschetz thimbles. 
These objects are thus expected to be part of a Fukaya category for stable GC manifolds, which we hope to develop in future work and which would allow the application of Floer theory techniques to a larger class of manifolds. 
\end{abstract} 

\tableofcontents 

\section{Introduction} 
Generalized complex geometry (for comprehensive references, see \cite{Gualtieri2003, Hitchin2010}) studies integrable complex structures on \emph{exact Courant algebroids} on a smooth manifold $M$, in the simplest instance complex structures on $TM\oplus T^*M$ with a certain integrability condition. 
This type of structure has sparked ongoing interest by virtue of including both complex and symplectic structures as examples, while generic generalized complex structures interpolate between the two: Each generalized complex structure induces a Poisson structure of varying rank, which determines symplectic leaves, to which there is then a transverse complex structure. 

The first and simplest examples of manifolds which admit neither a complex nor a symplectic, but a generalized complex structure, arise in the context of \emph{stable generalized complex structures} (systematically studied in \cite{Cavalcanti2015a}): These are generically symplectic, but the generalized complex structure degenerates on a real codimension-2 submanifold. 
As established in \cite{Cavalcanti2015a}, such generalized complex structures can be described in terms of so-called \emph{elliptic symplectic forms}, which are symplectic outside of and exhibit a particular type of singularity at the codimension-2 degeneracy locus. 

The main focus of this paper is on \emph{Lagrangian branes with boundary}, a new class of submanifold with boundary for stable generalized complex manifolds that is not included in the previously studied class of generalized complex branes. Lagrangian branes with boundary are generically Lagrangian in the bulk of the manifold and intersect the degeneracy locus cleanly in their boundary. 
Since they are Lagrangian with respect to the elliptic symplectic form and arise as Lefschetz thimbles in the stable generalized complex version of Lefschetz fibrations, they are expected to appear in the construction of a generalisation of the Fukaya category for stable generalized complex manifolds which considers non-compact Lagrangians, a generalisation of the Wrapped Fukaya category (see for example \cite{Abouzaid2007}). The generalisation of Floer theory and the construction of a Wrapped Fukaya category for (certain) stable generalized complex manifolds are aims for future work. 

In order to find a normal form for the local neighbourhoods Lagrangian branes with boundary and to study their deformations, it turns out to be useful to relate stable generalized complex manifolds to so-called \emph{logarithmic symplectic manifolds} via a real oriented blow-up of the codimension-2 degeneracy locus: Such manifolds carry a Poisson structure which is symplectic, i.e. non-degenerate, except on a hypersurface, in this case a boundary. 

\paragraph{\textbf{Summary of the paper.}} After summarising prerequisite results in section \ref{sec:b}, we give a definition for Lagrangian branes with boundary (section \ref{sec:bwb}). 
We then turn to Lagrangians in logarithmic symplectic manifolds, which intersect the degeneracy locus  transversely in a hypersurface: For such submanifolds, we can prove a Lagrangian neighbourhood theorem (section \ref{sec:lnbhd}). Section \ref{sec:symplecto} establishes the notion of symplectomorphism for logarithmic and elliptic symplectic structures and defines flux homomorphisms. 
The flux homomorphism for log symplectic manifolds allows us to study small deformations of log Lagrangians within a tubular neighbourhood:
Up to Hamiltonian isotopy, these are given by the first \emph{logarithmic cohomology} with respect to the hypersurface given by the intersection with the singular locus.

In section \ref{sec:blow} we prove that a real oriented blow-up on the codimension-2 degeneracy locus of a stable generalized complex manifold naturally produces a log symplectic manifold with boundary, give conditions for a converse blow-down of a log-symplectic to a stable generalized complex manifold, and establish a correspondence between Lagrangian branes (with and without boundary) and log Lagrangians which intersect the boundary transversely. 
Using this result and those from sections \ref{sec:lnbhd} and \ref{sec:symplecto}, we find a standard local neighbourhood of Lagrangian branes with boundary (a so-called \emph{wedge neighbourhood}). We prove that small deformations of Lagrangian branes with boundary up to Hamiltonian isotopy are given by the first logarithmic cohomology of the brane. 
Sections \ref{sec:logehr}, \ref{sec:gclef} and \ref{sec:thimb} consider Lefschetz thimbles in the generalisation of Lefschetz fibrations to the logarithmic and elliptic symplectic setting (first studied in \cite{Cavalcanti2016,Cavalcanti2017,Behrens2017}), which intersect the degeneracy locus. Many such thimbles are disk-shaped Lagrangian branes with $S^1$-boundary.

Section \ref{sec:hol} considers specific examples of complex surfaces which carry a holomorphic Poisson structure that is stable generalized complex. Lagrangian branes with boundary are never complex submanifolds. We do however provide examples where \emph{outside} the degeneracy locus, such branes can be deformed into complex curves via Hamiltonian isotopies, which are in some cases non-algebraic. 

\paragraph{\textbf{Acknowledgements.}} 
The work presented in this text was done under the supervision of Professor Marco Gualtieri at the University of Toronto as part of a stay as a Visiting International Research Student. The author extends sincere thanks to Marco Gualtieri for his supervision and guidance throughout the course of this project. I further express my gratitude to the University of Toronto for their hospitality during the academic year 2016/17, as well as Trinity College, Cambridge, and the Rouse-Ball Travelling Scholarship for Mathematics for providing funding, which allowed for this productive research visit. \\
Further thanks to Geoffrey Scott for useful discussions about log symplectic geometry. 
During the completion of this work, the author was supported by an STFC Studentship and a Graduate Studentship from Trinity College, Cambridge. Since October 2018, the author is supported by the long-term structural funding Methusalem grant of the Flemish Government. 

\section{Basic definitions and preliminary results} \label{sec:b}
Throughout this text, let $M^{2n}$ be an even-dimensional smooth manifold. Let $H\in \Omega^3_{\operatorname{cl}}(M)$. 

\subsection{Generalized complex geometry} 
\begin{defn}\textbf{\cite{Gualtieri2003, Hitchin2003}} A \emph{generalized complex structure} $\mathcal{J}$ on $M$ is $\mathcal{J}\in \End(TM\oplus T^*M)$ s.t.  $\mathcal{J}^2 =-\mathbbm{1}$, $\mathcal{J}$ is orthogonal with respect to the natural symmetric pairing\\ $\left<X+\xi,Y+\eta\right>=\eta(X)+\xi(Y)$, 
and the $+i$-eigenbundle of $\mathcal{J}$, $L\subset (TM\oplus T^*M)\otimes \Cnum$, is integrable with respect to the Courant-Dorfman bracket 
\[ \lb X+\xi,Y+\eta\rb = [X,Y]+L_X \eta - i_Y\dx \xi + i_Y i_X H \] 
\end{defn} 
Equivalently, generalized complex structures can be characterised directly by their $+i$-eigenbundle $L$, or by their so-called \emph{canonical bundle} $K\subset \wedge^{\bullet}T^*_{\Cnum}M$ (see e.g. \cite{Gualtieri2003, Cavalcanti2015a}). 

Any generalized complex structure $\mathcal{J}: \DTM \rightarrow \DTM$ induces a Poisson structure on $M$ as follows: 
\begin{prop} (see \cite{Gualtieri2003}.) The bivector
$\pi = \operatorname{pr}_{TM}\circ \mathcal{J}\vert_{T^*M}$ is Poisson. 
\end{prop} 
We call $\operatorname{Type}(\mathcal{J}) = n-1/2 \operatorname{rank}\pi $ the \emph{type} of the generalized complex structure. Generalized complex structures of type $0$ are equivalent to symplectic structures, while generalized complex structures of maximum type $n$ correspond to complex structures. In general, at every point a generalized complex structure of type $k$ is equivalent to the direct sum of a complex structure of complex dimension $k$ and a transverse symplectic structure of real dimension $2n-2k$ (see \cite{Gualtieri2003}). 

\subsection{Complex and elliptic divisors, stable generalized complex structures} 
There are generalized complex structures which are generically, but not everywhere, symplectic: Namely, their associated Poisson structures are symplectic structures almost everywhere, but change type along a lower dimensional submanifold. These so-called \emph{stable generalized complex structures} are studied systematically in \cite{Cavalcanti2015a}. All results and definitions presented in this subsection can be found in this reference. 

They are similar (and, as we show in section \ref{sec:blow}, intimately connected) to log symplectic manifolds: Poisson manifolds where the $n$-th power of the Poisson bivector $\pi^n\in \Gamma(\wedge^{2n} TM)$ vanishes transversely, and thus $\pi$ is non-degenerate outside a smooth codimension-1 submanifold. Such structures are studied in detail for example in \cite{Guillemin2014}, and they turn out to be equivalent to symplectic structures for the so-called log tangent bundle, a natural Lie algebroid associated to any codimension-1 submanifold. 

Stable generalized complex structures are formally described using the concept of \emph{complex divisors}, named because of their similarity to divisors in complex varieties: 

\begin{defn}(See \cite{Cavalcanti2015a}.)  \begin{enumerate}[label=(\roman*)]
\item A \emph{complex divisor} on a smooth manifold $M$ is a pair $D=(U,s)$ of a complex line bundle $U\rightarrow M$ and a section $s\in\Gamma(U)$ which intersects the zero section transversely. We also write $D=\{s=0\}$ and also call it the complex divisor. ($D\subset M$ is a smooth codimension-2 submanifold.) 
\item The \emph{vanishing ideal} associated to $D$ is 
\[ I_s= \operatorname{Im}\left(s: \Gamma(U^*)\rightarrow C^{\infty}_{\Cnum}(M)\right) \] 
\item The \emph{complex logarithmic tangent bundle} $T_{\Cnum}M(-\log D)$ associated to $D$ is the smooth vector bundle whose sections are
\[Z\in \Gamma(T_{\Cnum}M) \text{ s.t. } Z(I_s)\subset I_s \] 
(These form a locally free sheaf and are thus indeed the sections of a smooth vector bundle.) 
\end{enumerate} 
\end{defn} 
In fact, $T_{\Cnum}M(-\log D)$ inherits the Lie bracket from $T_{\Cnum}M$, as well as an anchor $a:T_{\Cnum}M(-\log D)\rightarrow T_{\Cnum} M$ and is thus a (complex) Lie algebroid. Its dual bundle is $T^*_{\Cnum}M(\log D)$, and there is a differential complex of complex logarithmic forms 
\[ \dx : \Gamma(\wedge^{k} T^*_{\Cnum}M(\log D)) \rightarrow \Gamma(\wedge^{k+1}T^*_{\Cnum}M(\log D)) \] 

Any complex divisor defines what is called an \emph{elliptic divisor} $(U\otimes \bar{U},s\otimes \bar{s})$, a pair of a \emph{real} line bundle on $M$ with a section that vanishes critically on the smooth codimension-2 submanifold $D$. 
\begin{defn} 
An \emph{elliptic divisor} is a pair $(R,q)$ of a real line bundle $R\rightarrow M$ with section $q\in \Gamma(R)$ which vanishes critically on a smooth codimension-2 submanifold $D \subset M$ s.t. the \emph{normal Hessian} of $q$ along $D$ is positive definite. The vector fields preserving the associated ideal $I_q=\operatorname{Im}(q:R^*\rightarrow C^{\infty}(M))$ form a locally free sheaf and are thus the sections of a smooth real vector bundle $TM(-\log\abs{D})$, the \emph{elliptic tangent bundle} associated to $(R,q)$. 
\end{defn} 
Just like for real logarithmic forms, it makes sense to consider the symplectic forms for the elliptic Lie algebroid $TM(-\log \abs{D})$: 
\begin{defn} 
An \emph{elliptic symplectic form} $\omega\in \Gamma(\wedge^2 T^*M(\log \abs{D}))$ is s.t. $\omega$ is non-degenerate as a two-form on $TM(-\log \abs{D})$ and $\dx \omega =0\in \Gamma(\wedge^3 T^*M(\log \abs{D})$. 
\end{defn} 
Clearly, $\pi = \omega^{-1}$ defines a Poisson structure on $M$ that is non-degenerate on $M\setminus D$ and has lower rank on $D$. 

Any elliptic divisor whose vanishing locus $D$ is co-oriented is in fact of the form $(U\otimes \bar{U},s\otimes \bar{s})$ with $(U,s)$ a complex divisor, and $(U,s)$ is unique up to isomorphism. There is thus a Lie algebroid morphism 
\[ \iota: TM(-\log\abs{D})\otimes \Cnum \rightarrow T_{\Cnum}M(-\log D) \] 

\paragraph{\textbf{Residues of logarithmic and elliptic forms}} For both logarithmic (real and complex) and and elliptic differential form there are notions of residue; the residue of such a form is always a smooth form on the degeneracy locus of smaller degree. 

For a (real or complex) logarithmic form $\alpha\in \Omega_{(\Cnum)}^k(M,\log Z)$, where the degeneracy locus $Z$ is locally given by the (real or complex) defining function $f$, the residue map is: 
\[ \res:  \Omega_{(\Cnum)}^k(M,\log Z)\rightarrow \Omega^{k-1}_{(\Cnum)}(Z), \alpha \mapsto \iota_Z^*\left(i_{f\party{}{f}} \alpha\right),\]
which can be shown to be independent of $f$. 

For an elliptic form $\alpha \in \Omega^k(M,\log \abs{D})$, where the elliptic divisor is locally given by the radial function $r^2$, and the corresponding angular coordinate for the normal bundle is $\theta$,  the \emph{elliptic residue} is 
\[ \res_e: \Omega^k(M,\log \abs{D})\rightarrow \Omega^{k-2}(D), \alpha \mapsto \iota^*_D\left(i_{r\party{}{r}} i_{\party{}{\theta}} \alpha\right)\] 
If $\res_e \alpha=0$, we can define the \emph{radial residue} 
\[ \res_r: \Omega^k(M,\log \abs{D})\rightarrow \Omega^{k-1}(D), \alpha\mapsto \iota^*_D\left(i_{r\party{}{r}} \alpha\right)\]
For details see \cite{Cavalcanti2015a}. 

\paragraph{\textbf{Stable generalized complex structures}} Recall that a generalized complex structure is uniquely defined by its canonical bundle $K\subset \wedge^{\bullet} T^*_{\Cnum} M$. Consider the section $s\in \Gamma(K^*)$ which projects any $\rho \in K_p, p\in M$ to its degree-zero-component: 
\[ \left<\rho,s_p\right>:= \rho_0 \in \Rnum \] 
\begin{defn} 
A \emph{stable generalized complex structure} is one where $D=(K^*,s)$ is a complex divisor, which we then call the \emph{anticanonical divisor}. By abuse of notation, we also write $D=\{s=0\}$ and call it the anticanonical divisor. 
\end{defn} 
\begin{thm}\textbf{(Theorem 3.2 in \cite{Cavalcanti2015a})} 
Any stable generalized complex structure $\mathcal{J}$ on $(M,H)$ defines a complex log form $\sigma=B+i\omega$ with $\dx \sigma=a^* H$ and $\omega$ non-degenerate for the anticanonical divisor $D=(K^*,s)$. (Such a form is called \emph{complex log symplectic}.) Conversely, given a complex divisor $D$ and a complex log symplectic form $\sigma$ for a particular pair $(M,H)$, we can construct a stable generalized complex structure. These two assignments are inverse to each other. In this correspondence, any local trivialisation of the canonical line bundle $K$ satisfies 
\[ a^*\rho= \rho_0 e^{\sigma}. \]
\end{thm} 
Let $\Gamma_{\sigma}\subset T_{\Cnum}M(-\log D) \oplus T^*_{\Cnum}M(\log D)$ denote the graph of $\sigma$. Then this correspondence is 
\[ L_{\mathcal{J}} = a_* \Gamma_{\sigma} := \{ a(X)+\eta \in \mathbbm{T}_{\Cnum}M \vert X+ a^*\eta \in \Gamma_{\sigma}\},  \] 
Let $L_{\mathcal{J}}$ denote the $+i$-eigenbundle of $\mathcal{J}$. 
Under B-transforms $\mathcal{J} \mapsto e^{B'} \mathcal{J} e^{-B'} (B'\in \Omega^2_{\operatorname{cl}}(M))$, the log symplectic form $\sigma$ transforms as follows: 
 \[ \sigma \mapsto \sigma + B'=(B+B')+i\omega. \] 
\begin{thm}\textbf{(Theorem 3.7 in \cite{Cavalcanti2015a})} Let $M$ be a smooth manifold. The forgetful map taking a pair $(\mathcal{J},H)$ of a closed 3-form $H$ and stable generalized complex structure $\mathcal{J}$, which is integrable w.r.t. $H$, to the pair $(Q,\mathfrak{o})$ of the real Poisson structure $Q= \operatorname{pr}_{TM}\circ \mathcal{J}\vert_{T^*M}$ and the co-orientation $\mathfrak{o}$ of the anticanonical divisor $D$ defines a bijection between gauge equivalence classes of stable generalized complex structures (w.r.t. B-transforms) and elliptic symplectic structures $\omega=Q^{-1}$ with vanishing elliptic residue and co-oriented degeneracy locus. 
\end{thm} 
So a stable generalized complex structure $\mathcal{J}$ corresponds to a complex log symplectic form $\sigma$ whose imaginary part is $\omega$. As the imaginary part of a complex log form, $\omega$ has vanishing elliptic residue, and as we can see above, $\omega$ is invariant under B-transforms. A complex log form like $\sigma$ is determined by its imaginary part up to the addition of smooth 2-forms. 

Using these two theorems, stable generalized complex structures and elliptic symplectic forms with vanishing elliptic residue are frequently treated interchangeably in this text. 

\subsection{Generalized complex branes in stable generalized complex manifolds} 
Generalized complex branes are a class of natural submanifolds of a generalized complex manifold. 
There are several similar, but non-equivalent definitions for branes carrying complex line bundles, compare for example \cite{Gualtieri2011} and \cite{Collier2014}. These two definitions both involve a complex line bundle supported on the submanifold; there is a simpler definition involving just the submanifold equipped with a smooth two-form which includes both concepts.  This definition has previously been used by \cite{Cavalcanti2015a,Cavalcanti2009} and others: 

\begin{defn} \label{def:gcbrane}
A \emph{generalized complex brane} in a generalized complex manifold $(M,H,\mathcal{J})$ is a pair $(Y,F)$ of a submanifold $\iota: Y\hookrightarrow M$ and a two-form $F\in \Omega^2(Y)$ such that 
\begin{itemize} 
\item $dF=\iota^*H$ 
\item $\tau_F=\{X+\xi \in TY\oplus T^*M\vert_Y \text{ s.t. } \iota^*\xi=i_X F\}\subset \DTM\vert_Y$ is preserved by $\mathcal{J}$: 
\[ \mathcal{J}(\tau_F)=\mathcal{J} \] 
\end{itemize} 
\end{defn} 

In the standard examples of symplectic and complex manifolds, generalized complex branes are known: Complex branes are precisely complex submanifolds equipped with closed $(1,1)$ forms. Half-dimensional branes in symplectic manifolds are Lagrangian submanifolds with zero two-forms (or closed two-form after $B$-transform). There are also higher-dimensional coisotropic branes called \emph{coisotropic A-branes}. Details can for example be found in \cite{Gualtieri2011}. 

Now, since stable generalized complex manifolds are generically symplectic, their half-dimensional branes will be generically Lagrangian w.r.t. the elliptic symplectic form $\omega$. The aspects that sets them apart from branes in pure symplectic manifolds are their intersection with the anticanonical divisor $D$, as well as generically non-zero $F=\iota^*B$ (where $\sigma = B+i\omega,\ \iota$ inclusion of brane). 

\begin{prop}\textbf{(Proposition 3.42 in \cite{Cavalcanti2015a})} Any submanifold $L\subset M$ in a stable generalized complex manifold $(M,\mathcal{J})$ which is transverse to the anticanonical divisor $D$ and Lagrangian for the elliptic symplectic structure underlying $\mathcal{J}$ inherits a smooth 2-form $F=\iota^*B$ making it into a generalized complex brane. 
\end{prop} 
Note that because $L\pitchfork D$, the elliptic divisor on $M$ pulls back to form an elliptic divisor on $L$, so $\omega$ pulls back to $L$ as an elliptic form, and it makes sense to demand that this pullback be zero. 

The elliptic cotangent bundle of any manifold equipped with an elliptic divisor carries a natural elliptic symplectic structure, defined in the same way as for the ordinary cotangent bundle: 
$T^*M(\log \abs{D})$ has  a pullback elliptic divisor with singular locus $T^*M(\log \abs{D})\vert_D$, and the natural elliptic symplectic form is the derivative of the tautological elliptic one-form. 
Thus there is the following natural Lagrangian neighbourhood theorem for Lagrangian generalized complex branes intersecting the degeneracy locus transversely: 
\begin{thm}\label{thm:lnbhd1}\textbf{(Theorem 3.38 in \cite{Cavalcanti2015a})} If $(M,D,\omega)$ is an elliptic symplectic manifold and $L$ a compact Lagrangian submanifold transverse to $D$, there exists a tubular neighbourhood of $L$ which is elliptic symplectomorphic to a tubular neighbourhood of the zero section in $T^*L(\log \abs{L\cap D})$ equipped with the natural elliptic symplectic form on the elliptic cotangent bundle. 
\end{thm} 

\section{Lagrangian branes with boundary in stable generalized complex manifolds} \label{sec:bwb}
In this section we introduce and investigate the principal objects in the focus of this text: Lagrangian branes with boundary. 

We have just presented results on generically Lagrangian submanifolds $L$ of stable generalized complex manifolds $(M^{2n}, D^{2n-2},\sigma=B+i\omega)$ which intersect the anticanonical divisor $D$ transversely, generalized complex branes in the sense of Definition \ref{def:gcbrane}. 
Now we instead consider generically Lagrangian submanifolds with boundary $(L^n,(\partial L)^{n-1})$ of $(M^{2n},D^{2n-2},\sigma)$ which intersect $D$ cleanly in their boundary. This implies that the intersection is not transverse, and in fact these submanifolds are \emph{not} generalized complex branes. 
In this section, we introduce \emph{wedge neighbourhoods} of a brane with boundary and make sense of the pullback of elliptic differential forms to logarithmic differential forms on such a brane. 

\begin{defn} \label{def:bb}
An $n$-dimensional submanifold with boundary $\iota_L: (L,\partial L)\hookrightarrow (M, D, \omega)$ is a \emph{Lagrangian brane with boundary} if  
\begin{equation} L\cap D = \partial L \text{ and } T(L\cap D)=TL\cap TD\vert_{L\cap D} \text{ (clean intersection) } \end{equation}
and 
\[ \iota_L^* \omega =0\] 
\end{defn} 
Of course a priori $\iota^* \omega$ is only defined outside $D$, but Proposition \ref{prop:incl} illustrates how to make sense of this expression on all of $M$. To prove this proposition, we first consider natural neighbourhoods of submanifolds with boundary inside the degeneracy locus: 

Let $(Y,\partial Y)\subset (M,D)$ be any submanifold with boundary in a manifold equipped with a complex (and thus an induced elliptic) divisor, intersecting $D$ cleanly in its boundary. For such manifolds, which include Lagrangian branes with boundary, there is a natural notion of local neighbourhood, although these neighbourhoods are not open submanifolds of $M$, i.e. not tubular neighbourhoods in the conventional sense: 

We can choose a tubular neighbourhood of $D$ in $M$ in such a way that a collar neighbourhood of $\partial Y$ in $Y$ defines a rank-1 subbundle in $ND$. 
\begin{defn} \label{defn:wedge}
Let $V\subset D$ be a tubular neighbourhood of $\partial Y$ in D, isomorphic to $ND$. Consider the restriction $ND\vert_V$,  a trivial rank-2 bundle. In every fibre over $\partial Y$, pick a \emph{wedge} 
\[ W^2 := \Rnum_{>0}\times \Rnum_{>0} \cup \{(0,0)\} \] 
around the 1-dimensional subspace defined by $Y$, where the tip of of wedge is the base point. Since $ND\vert_{\partial Y}$ is trivial, such a choice can be consistently made across $\partial Y$, and extend to $ND\vert_V$, to glue together to a smooth $W^2$-bundle over $V$. Let $\hat{V}$ be the image of this $W^2$-bundle inside the tubular neighbourhood of $D$ in $M$. \\ 
A \emph{wedge neighbourhood} of $(Y,\partial Y)\subset (M,D)$ consists of the smooth gluing of such a $W^2$-neighbourhood $\hat{V}$ with a tubular neighbourhood of $Y\setminus (\partial Y \times [0,1))$ inside $M\setminus D$.  
\end{defn} 
Note that such a space is not a smooth manifold, but instead has the following local type near $\partial Y$: Open neighbourhoods of points in $\partial Y$ inside the wedge neighbourhood are of the form $W^2\times \Rnum^{\dim M - 2}$. We call such spaces \emph{wedge manifolds} and equip them with a smooth structure: A map is smooth on $W^2\times \Rnum^{k-2}$ if it is smooth away from $\{0\}\times \Rnum^{k-2}$ and can be extended to a smooth map on some proper tubular neighbourhood of $\{0\}\times \Rnum^{k-2}$ in $\Rnum^k$. \\ 
In the case where the wedge neighbourhood is embedded in $M$ as above, it inherits its smooth structure from $M$. 
\begin{lem} \label{lem:smooth}
If $\iota_Y:Y\hookrightarrow M$ is a submanifold with (smooth) boundary in a manifold with a complex and induced elliptic divisor, such that  $\partial Y= Y\cap D$ and \begin{equation} T(Y\cap D)=TY\cap TD\vert_{Y\cap D}, \end{equation} we can, inside a wedge neighbourhood of an open neighbourhood of $\partial Y$ in $Y$, choose the polar coordinates $(r,\theta)$ in such a way that $r\party{}{r}$ is tangent to $Y$ in an open neighbourhood of $\partial Y$. 
\end{lem} 
\begin{proof} 
Since a complex divisor is given by a transversely vanishing section of a complex line bundle, we can locally describe it by a complex function $z=r e^{i\theta}=a+ib$, which is however only defined up to multiplication by a nowhere vanishing complex function $g=\abs{g} e^{i\sigma}$, where $\abs{g}$ is a smooth map from an open neighbourhood of $D$ to the positive real numbers and $\sigma$ a smooth map to $S^1$. 
In addition to the polar coordinates $(r,\theta)$ we choose coordinates $y_3,\dots,y_{2n}$ to describe a full tubular neighbourhood of $D$. 
By multiplying $z$ by $e^{\sigma}$, where $\sigma$ only depends on the $y^i$, we can always rotate $z$ so that $\left.\party{}{a}\right.\vert_{\partial N}$ is tangent to $Y$ and inward-pointing. Then in a small neighbourhood of $D$, one of the equations determining $Y$ is
\[ \theta= \lambda(a, y^i), \] 
where $\lambda$ is a smooth function with $\lim_{a\to 0}\lambda(a,y^i)=0$, so $\abs{\lambda}<\frac{\pi}{4}$ in some neighbourhood of $D$. (If we choose the $y^i$ correctly, the other equations determining $Y$ are $r=\sqrt{a^2+\beta^2(a,y^i)}$, $\beta$ a smooth function, of the form $y^j=0$ for $j>k$, and $\lambda$ only depends on $y^i, i\leq k$.) 
Clearly $\lambda$ can be extended to a small neighbourhood of $Y$ simply as a constant function in $b$. On the intersection of the tubular neighbourhood of $D$ and some wedge neighbourhood of $\{b=0\}$ the transformation 
\[ z \mapsto e^{-i \lambda(a,y^i)}z, \text{ i.e } \theta \mapsto \theta'=\theta - \lambda(a, y^i) \] 
is a diffeomorphism which takes $Y$ to $\{\theta'=0\}$, so $\party{}{a'}, a'=r\cos (\theta')$ is tangent to $Y$ in that neighbourhood, as is $\left.r\party{}{r}\right\vert{Y}=\left.a' \party{}{a'}\right\vert_Y$. 
\end{proof} 

\begin{prop} \label{prop:incl}
Let $\iota_Y:Y\hookrightarrow M$ be a submanifold with (smooth) boundary in a manifold with a complex and induced elliptic divisor, such that  $\partial Y= Y\cap D$ and \begin{equation} T(Y\cap D)=TY\cap TD\vert_{Y\cap D}. \end{equation}
Then the elliptic divisor $(R,q)$ on $M$ induces morphisms 
\begin{align}  
\iota_{Y,*}: TY(-\log Y\cap D) &\rightarrow TM(-\log \abs{D})\vert_Y \\ 
\iota_Y^*: \Omega^{k}(M,\log \abs{D}) &\rightarrow \Omega^{k}(Y,\log N\cap D),  \end{align}
where $TY(-\log Y\cap D)$ is the real logarithmic tangent bundle for $Y\cap D$ inside $Y$, and $\Omega^{k}(Y,\log Y\cap D)$ the real logarithmic differential forms. 
\end{prop} 
This ensures that that the pullback $\iota^*_L \omega$ in Definition \ref{def:bb} makes sense. 
\begin{proof} 
It is sufficient to construct $\iota_{Y,*}: TY(-\log Y\cap D)\rightarrow TM(-\log \abs{D})$ on $\partial Y= Y\cap D$ and to show that this extends the ordinary pushforward map on the interior smoothly. 

Locally, the elliptic divisor is given by a function $r^2$ with $D=\{r=0\}$ (As seen above, $r^2$ is only determined up to multiplication with a positive real function). 

The function $x=\sqrt{\iota^* r^2}$ is smooth on $(Y,\partial Y)$, more particularly, it is a defining function for $\partial Y$.
 It is a well-known fact from log geometry that $\left.x \party{}{x}\right\vert_{\partial Y}\in \Gamma(TY(-\log \partial Y)\vert_{\partial Y})$ is independent of the choice of defining function. Similarly, in elliptic geometry $\left.r \party{}{r}\right\vert_{D}$ only depends on the elliptic divisor, not the function $r^2$. So we set: 
\[ \iota_*\left(x\party{}{x}\biggm\vert_{\partial Y}\right) = \left. r \party{}{r}\right\vert_{Y\cap D}\] 
Then $\iota_*: TY(-\log \partial Y)\rightarrow TM(-\log \abs{D})$ is smooth: If we choose $(r,\theta)$ as in Lemma \ref{lem:smooth}, $x\party{}{x}$ extends to $r\party{}{r}$ on a wedge neighbourhood of a neighbourhood of $\partial Y$ in $D$. This allows us to see that $i_*$ clearly pushes log vector fields on $Y$ forward to restrictions of elliptic vector fields to $Y$ in a smooth manner. 
\end{proof} 

\begin{ex} \textbf{(Standard local example)} 
According to Theorem 3.21 in \cite{Cavalcanti2015a}, if $M$ is a stable generalized complex manifold with anticanonical divisor $D$, the associated complex logarithmic symplectic form can be written in local coordinates $(w,z, q_3,\dots,q_n,p_3,\dots,p_n )$ ($w,z$ complex coordinates, $q_i,p_i$ real) around any point in $D$ as 
\[ \sigma = \frac{\dx w}{w}\wedge \dx z + i \sum_j \dx p_j \wedge \dx q_j \]
If we write $w=r e^{i\theta}, z=x+iy$, we obtain 
\[ \omega= \operatorname{Im} (\sigma)= \frac{\dx r}{r}\wedge \dx y + \dx \theta \wedge \dx x + \sum_j \dx p_j \wedge \dx q_j. \] 
Then the following planes in $\Rnum^{2n}$ all define Lagrangian branes with boundary in the sense of Definition \ref{def:bb}: 
\[ \{ y=\text{const},\ \theta=\text{const},\ \iota^*(\dx p_j \wedge \dx q_j) =0 \text{ (e.g. }p_j=\text{const)}\} \] 
\end{ex} 

\begin{prop} \label{prop:resb}
If the stable generalized complex structure is given by the complex log symplectic form $\sigma=B+i\omega$, a Lagrangian brane with boundary carries a natural logarithmic two-form $F=\iota^* B$ with non-vanishing residue. 
\end{prop} 
(This proposition will be proved at the end of section \ref{sec:nbhdbdy} of this paper.)

Hence, Lagrangian branes with boundary are \emph{not} generalized complex branes. This text will argue that they should nonetheless be considered when studying submanifolds of stable generalized complex manifolds, and show how they fit into a general framework of branes in stable generalized complex manifolds. Towards this aim, we will establish a correspondence of stable generalized complex manifolds and log symplectic manifolds, as well as their Lagrangian submanifolds. 

\section{Lagrangian neighbourhood theorem for log symplectic manifolds} \label{sec:lnbhd}
Let $(M,Z,\omega)$ be a real logarithmic symplectic manifold. We can prove a Lagrangian neighbourhood theorem for compact Lagrangians that intersect the singular locus transversely, employing the same techniques as in the proof of Weinstein's original Lagrangian neighbourhood theorem and the version for stable generalized complex manifolds (see Theorem \ref{thm:lnbhd1} above and \cite{Cavalcanti2015a}). As far as we are aware, this proof has not previously appeared in the literature. 
\begin{prop} 
If $\iota_Y: Y\hookrightarrow M$ is a submanifold which intersects $Z$ transversely, there are induced morphisms 
\begin{align*}  \iota_{Y,*}:  TY(-\log Y\cap Z) &\rightarrow TM(-\log Z)  \\ 
\iota_Y^*: \Omega^{\bullet}_M(\log Z)&\rightarrow \Omega^{\bullet}_Y(\log Y\cap Z) 
\end{align*} 
\end{prop} 
\begin{proof} 
This proceeds exactly like the proof for Proposition \ref{prop:incl}: Again, we only need to consider the pushforward of $\left. x\party{}{x}\right\vert_{Y\cap Z}$, where $x$ is a defining function for $Y\cap Z$. Obviously, since $Y\pitchfork Z$, any defining function for $Z$ on $M$ will provide one for $Y\cap Z$ in $N$ via pullback. Assume $x=\iota^*\tilde{x}$. Then we can define the pushforward 
\[ \iota_{Y,*}\left(\left. \party{}{x}\right\vert_{Y\cap Z}\right) = \left. \tilde{x}\party{}{\tilde{x}}\right\vert_{Y \cap Z}, \]
which is well-defined and smooth: $\tilde{x}$ can be chosen in such a way that $\tilde{x}\party{}{\tilde{x}}$ is tangent to $Y$.  
The definition of the pullback for logarithmic forms is then obvious. 
\end{proof} 

\begin{thm}\textbf{(Lagrangian neighbourhood theorem for log symplectic manifolds)} \label{thm:lognbhd}
Let $(M,Z,\omega)$ as above and $\iota_L: L\hookrightarrow M$ a compact Lagrangian submanifold which intersects the degeneracy locus $Z$ transversely. 

Then there is a neighbourhood $(U,U\cap Z)$ of $L$ in $M$ which is isomorphic to a neighbourhood of the zero section in $T^*L(\log L\cap Z)$, i.e. there exists a diffeomorphism onto its image 
\[ \phi: (U,U\cap Z) \rightarrow T^*L(\log L\cap Z)\] 
such that $\phi^*(\omega_0)=\omega$, where $\omega_0$ is the standard log symplectic form on $T^*L(\log L\cap Z)$, and $(\phi(U),\phi(U\cap Z))$ is a tubular neighbourhood of $(L,L\cap Z)$. 
\end{thm} 

\begin{proof} The proof proceeds exactly like that of the original Weinstein Lagrangian neighbourhood theorem, see for example \cite{CannasdaSilva2001}. In this case, we start by choosing a tubular neighbourhood for $(L,L\cap Z)$ whose intersection with $Z$ is a tubular neighbourhood for $L\cap Z$. 

\textbf{Claim:} The cokernel of $\iota_{L,*}$ is $TM(-\log Z)\vert_L/\operatorname{Im}(\iota_{L,*}) \cong NL$.
\begin{proof}We consider the sheaves of sections and show that they are isomorphic as locally free sheaves. 
The sheaf $\Gamma\left(TM(-\log Z)\vert_L/TL(-\log L\cap Z)\right)$ includes into $\Gamma(NL)=\Gamma(TM\vert_L/TL)$ via the anchor. \\
The inverse map is as follows: Let $X+\Gamma(TL)\in \Gamma(TM\vert_L/TL)$. We have assumed $L\pitchfork Z$, so 
\[(TL+TZ)\vert_{L\cap Z}= TM \vert_{L\cap Z},\] 
i.e. we can write $X\vert_{L\cap Z}=X'+Y$, where $X'\in \Gamma(TL\vert_{L\cap Z}), Y\in \Gamma(TZ\vert_{L\cap Z})$. 
Extend $X'$ to a section on all of $L$. Then 
\[X-X' + \Gamma(TL) = X+\Gamma(TL),\text{ and }X-X' \in \Gamma(TM(-\log Z)\vert_L).\] 
Lastly, check that the class of $X-X'$ in $\Gamma(TM(-\log Z))\vert_L/TL(-\log L\cap Z)$ does not depend on the choice of $X'$ and its extension. 
\end{proof}

Since $L\pitchfork Z$, and $Z\subset M$ a codimension-1 submanifold, there is a tubular neighbourhood $U$ of $L$ in $M$ such that $U\cap Z$ is a tubular neighbourhood of $L\cap Z$ in $Z$. Now, since $TM(-\log Z)\vert_L/\operatorname{Im}(\iota_{L,*})\cong NL$ and $\iota_L^*\omega=0$, we obtain an isomorphism 
\[ \omega: TM(-\log Z)\vert_L/\operatorname{Im}(\iota_{L,*}) \rightarrow T^*L(\log L\cap Z), \]
which maps the tubular neighbourhood $(U,U\cap Z)$ of $(L,L\cap Z)$ to a tubular neighbourhood of the zero section in $T^*L(L\cap Z)$ in such a way that $U\cap Z$ gets mapped to the fibre over $L\cap Z$. 

Thus we can now view both $\omega$ and the natural log symplectic form on $T^*L(\log L\cap Z)$, $\omega_0$, as log symplectic forms on $(U,U\cap Z)$, both of which satisfy $\iota^*_L\omega= 0=\iota^*_L \omega_0$. 
The projection $p: TL(-\log L\cap Z)\rightarrow L$ induces an isomorphism on log cohomology $p^*:H^{\bullet}_L(\log L\cap Z) \rightarrow H^{\bullet}_U(\log U\cap Z)$. 
$p$ is homotopic to the identity on $U$, so $p^*$ is an isomorphism on cohomology: Since $p\circ \iota=\operatorname{Id}_L, \iota\circ p \sim \operatorname{Id}_U$, $i^*: H^{\bullet}_U(\log U\cap Z) \rightarrow H^{\bullet}_L(L\cap Z)$ is the inverse of $p^*$ on cohomology. 
Thus, since $i^*_L\omega=i^*_L\omega_0=0$, $\omega,\omega_0$ are in the same log cohomology class on $U$ (in fact, both are trivial in cohomology on $U$). 
\[ \Rightarrow \omega - \omega_0 = \dx \alpha \text{ for some } \alpha \in \Omega^1(U,\log U\cap Z). \] 
We now consider the family of cohomologous closed log symplectic forms $\omega_t=t\omega+(1-t)\omega_0$ (these are non-degenerate on a small tubular neighbourhood of $L$ for all $t\in[0,1]$) and apply 
 the Moser argument: 
\[ X_t:= -\omega_t^{-1}(\alpha), \text{ where } \dx \alpha = \omega-\omega_0 \] 
is a well-defined logarithmic vector field, in particular it is smooth. We assume that $\iota^*_L\alpha=0$, which is clearly always possible. $L$ was assumed to be compact, so this time-dependent log vector field can be integrated to a family of diffeomorphisms $\psi_t, t\in (0,1)$ on a small neighbourhood of $L$ in $U$, which preserve $U\cap Z$. We have $\psi_t\vert_L: L\rightarrow L$, since $\iota_L^*\alpha=0$. Furthermore $\psi_0=\operatorname{Id}$, so: 
\begin{align*} 
\psi^*_t(\omega_t)&=\omega_0\ \forall t\in [0,1] 
\end{align*} 
Thus there is a neighbourhood of $(L,L\cap Z)$ in $(U,U\cap Z)$ with diffeomorphism $\psi_1^*(\omega)=\omega_0$, which proves the theorem. 
\end{proof} 

\begin{cor} 
If $L\subset M$ a compact Lagrangian such that $L\pitchfork Z$, each connected component of $L\cap Z$ lies inside a single symplectic leaf of $\omega^{-1}$ in $Z$ and is Lagrangian inside this leaf. 
\end{cor} 
\begin{proof} 
From the neighbourhood theorem, we obtain a tubular neighbourhood of $L$ with coordinates $(x^1,x^2,\dots x^n, y_1\dots,y_n)$ around a point of $L\cap Z$ such that $x_i$ are coordinates for $L$, $x_1$ a defining function for $Z$, $y_i$ are fibre coordinates for $T^*L(-\log L\cap Z)$, and where the log symplectic form is given by 
\[ \omega = \frac{\dx x^1}{x^1}\wedge \dx y_1 + \sum_{i>1} \dx x^i \wedge \dx y_i \] 
Clearly, the symplectic leaves of $\omega^{-1}$ are given by the integrable distribution $\ker (\res \omega)$. The intersection $L\cap Z$ in these coordinates is given by $y_i=0, x_1=0$ and thus clearly $T(L\cap Z)\subset \ker(\res \omega)=\ker \dx y_1$, so each connected component of $L\cap Z$ will lie inside a single symplectic leaf. 

The symplectic form on the symplectic leaves will clearly be given by 
\[ \omega'=\sum_{i>1} \dx x^i \wedge \dx y_i, \] 
so $L\cap Z$ will be Lagrangian inside the symplectic leaf. 
\end{proof} 

\section{Logarithmic and elliptic symplectomorphisms} \label{sec:symplecto}
First consider a compact log symplectic manifold $(M,Z,\omega)$. A \emph{diffeomorphism of the pair $(M,Z)$} is simply a diffeomorphism that preserves $Z$. These diffeomorphisms clearly form a subgroup of the diffeomorphism group of $M$ whose Lie algebra is precisely given by the logarithmic vector fields $\Gamma(TM(-\log Z))$. As a consequence, the pullback and push-forward of logarithmic forms and vector fields with respect to such \emph{logarithmic diffeomorphisms} are well-defined in a natural way, and a \emph{logarithmic symplectomorphism} is simply such a diffeomorphism $\phi$ which satisfies 
\[ \phi^*\omega= \omega. \] 

Now, for a compact elliptic symplectic manifold $(M,D,\omega)$, this is less obvious: For a general diffeomorphism of $\phi: M\rightarrow M$, and the chosen elliptic divisor $D=(R,s)$ on $M$, we can always consider the pullback divisor $\phi^*D=(\phi^*D,\phi^*s)$. This is isomorphic to the original divisor, and, up to isomorphism, gives rise to the same elliptic tangent and cotangent bundle. But the space of elliptic vector fields inside smooth vector fields will in general be different, even if $\phi$ preserves $D$.  

In order to compare the symplectic form before and after pullback with a diffeomorphism, we need the notion of elliptic vector field to stay the same, i.e. if $X\in\Gamma(TM(-\log\abs{D}))$, we need $\phi_*(X)\in\Gamma(TM(-\log\abs{D}))$ \emph{with respect to the original elliptic divisor}. 

\begin{prop}\label{prop:ellflow}
The flow $\phi_t$ of a time-dependent elliptic vector field $X_t \in \Gamma(TM(-\log \abs{D}))$ preserves the space of elliptic vector fields under push-forward, i.e. 
\[ Y\in \Gamma(TM(-\log\abs{D})) \Rightarrow (\phi_t)_* Y \in \Gamma(TM(-\log\abs{D})) \] 
The diffeomorphisms obtained in this manner form a subgroup of the identity component of the diffeomorphism group. 
\end{prop} 

This result follows from general Lie groupoid and Lie algebroid theory, which can for example be found in Chapter 3 of \cite{Mackenzie2005}. If $A$ is the Lie algebroid of the Lie groupoid $G$, there is a notion of exponential map allowing the integration of Lie algebroid sections to bisections of the Lie groupoid, which reaches the entire identity subgroupoid of $G$. $G$ acts adjointly on $A$ and the groupoid bisections induce diffeomorphisms -- in this case, this means that elliptic diffeomorphisms obtained as the flow of elliptic vector fields do indeed act on the original elliptic vector fields. 

So it makes sense to consider the diffeomorphisms obtained as the flow of (time-dependent) elliptic vector fields, corresponding to the identity component of the elliptic groupoid. Within this set, an \emph{elliptic symplectomorphism} is one that satisfies 
\[ \phi^*\omega = \omega. \] 
We call this subgroup of (the identity component of) all diffeomorphisms of $M$ the \emph{elliptic symplectomorphism group}. 

\subsection{The Flux homomorphism for log and elliptic symplectic manifolds} 
Analogously to ordinary symplectic manifolds, we can define a flux homomorphisms for log and elliptic symplectic manifolds to pick out Hamiltonian diffeomorphisms (i.e. the endpoints of  Hamiltonian isotopies) in the identity component of the log or elliptic symplectomorphism group respectively. 
To simplify the notation, in this subsection only, we will denote the elliptic and logarithmic objects in the same manner: The singularity locus is $D$, the elliptic or log symplectic form is $\omega$, the log or elliptic tangent bundle is $TM(-\log D)$, the elliptic or log cohomology is $H^{\bullet}(M,\log D)$, and so on. Where there is a difference between the logarithmic and the elliptic case, it will be specifically indicated. 
Denote by $\operatorname{Symp}_0(M,\omega)$ the identity component of the group of (log or elliptic) symplectomorphisms of a log or elliptic symplectic manifold $(M,D)$ with (log or elliptic) symplectic form $\omega$, and by $\widetilde{\Symp}_0(M,\omega)$ its universal cover. 

The results and proofs in this section closely follow \cite{McDuffSala1998} and \cite{Oh2015}, which develop this theory for ordinary symplectic manifolds. 

\begin{defn}
The \emph{flux homomorphism} is 
\[ \Flux: \widetilde{\Symp}_0(M,\omega) \rightarrow H^1(M,\log D),\ \Flux(\{\psi_t\})= \int_0^1 [i_{X_t}\omega]\dx t,  \] 
where $\{\psi_t\}, t\in [0,1]$ is a representative of a homotopy class of paths in $\Symp_0(M,\omega)$ with endpoint $\psi_1$, and $X_t$ its associated time-dependent log or elliptic vector field. 
\end{defn} 
\begin{thm} 
The flux homomorphism, as above, is well-defined, and a group homomorphism. 
\end{thm} 

To prove this theorem, we first establish the two following lemmas: 
\begin{lem}\textbf{(Banyaga's Lemma, \cite{Banyaga1978})} \label{lem:bany}
Let $\{\phi^s_t\}$ be a smooth two-parameter family of diffeomorphisms on $M$. Denote 
\[ X=X^s_t= \party{\phi^s_t}{t} \circ (\phi^s_t)^{-1},\ Y=Y^s_t= \party{\phi^s_t}{s} \circ (\phi^s_t)^{-1}. \] 
Then 
\[ \party{Y}{t}=\party{X}{s} + [Y,X] \]
\end{lem} 
This formulation of Banyaga's Lemma and a proof can be found in \cite{Oh2015}, Lemma 2.4.2. The following lemma generalises Lemma 2.4.3 in \cite{Oh2015} to log and elliptic forms: 
\begin{lem}
Let $\{\psi_t\},\{\psi'_t\} \in \widetilde{\Symp}_0(M,\omega)$ two paths from the identity to the same endpoint $\psi=\psi_1=\psi'_1$. Let $X_t,X'_t$ be the log/elliptic vector fields associated to these paths. \\
If $\{\psi_t\},\{\psi'_t\}$ are homotopic relative to their ends, then the log/elliptic one-form 
\[ \int_0^1 i_{X_t-X'_t}\omega \dx t \] 
is exact. 
\end{lem} 
\begin{proof} 
Let $\psi^s_t$ be a homotopy between $\psi_t,\psi'_t$ relative to $\{0,1\}$ i.e. 
\[ \psi^s_0= \Id_M,\ \psi^s_1 = \psi,\ 0\leq s\leq 1 \] 
Denote by $X^s_t$ the vector field associated to the path $\{\psi^s_t,0\leq t\leq 1\}$. \\ 
If we can prove that the log/elliptic one-form 
\[ \frac{\dx}{\dx s} \int_0^1 i_{X^s_t}\omega \dx t \] 
is exact for all $s$, the result follows. 

We have already defined $X^s_t$ as in Lemma \ref{lem:bany}; define $Y^s_t$ accordingly as well. All $\psi^s_t$ are log/elliptic symplectomorphisms, so the vector fields $X^s_t, Y^s_t$ are symplectic vector fields. Compute: 
\[ \frac{\dx}{\dx s} \int_0^1 i_{X^s_t}\omega \dx t= \int_0^1 \party{}{s}(i_{X^s_t}\omega)\dx t = \int_0^1i_{\party{X^s_t}{s}}\omega \dx t, \]
which we can rewrite using Lemma \ref{lem:bany}: 
\[   \frac{\dx}{\dx s} \int_0^1 i_{X^s_t}\omega \dx t= \int_0^1 i_{\left(\party{Y^s_t}{t} + [X^s_t,Y^s_t]\right)}\omega \dx t.\]
Now, the first part of the integral is simply 
\[ i_{Y^s_1-Y^s_0}\omega =0\ \forall s, \] 
since $\psi^s_1=\psi\ \forall s, \psi^s_0=\Id_M\ \forall s$. As the Lie bracket of two log/elliptic symplectic vector fields, $[X^s_t,Y^s_t]$ is a log/elliptic Hamiltonian vector field (proof proceeds exactly as in the ordinary symplectic case) and thus the second term is exact for all $s$. 
\end{proof} 

This shows that $\Flux: \widetilde{\Symp}_0(M,\omega) \rightarrow H^1(M,\log D)$ is well-defined. 

The group homomorphism property follows easily from the fact that if the isotopy of symplectomorphisms $\{\phi_t\}$ is generated by $X_t$, and $\{\psi_t\}$ by $Y_t$, $\psi_t \circ \phi_t$ is generated by $Y_t + (\psi_t)_* X_t$. Furthermore, $\phi_* X - X$ is Hamiltonian whenever $X$ is a symplectic vector field and $\phi$ a symplectomorphism. 

\begin{thm}\textbf{(Compare Theorem 10.12 in \cite{McDuffSala1998})}  
Let $\psi \in \Symp_0(M,\omega)$. $\psi$ is a log/elliptic Hamiltonian symplectomorphism if there exists a symplectic isotopy $\psi_t$ with $\psi_0=\operatorname{Id}_M, \psi_1=\psi$ such that 
\[ \Flux(\{\psi_t\}) = 0.\] 
Conversely, given a symplectic isotopy $\{\psi_t\}$ with $\Flux(\{\psi_t\})=0$, it is homotopic (with fixed endpoints) to a Hamiltonian isotopy. 
\end{thm} 
The proof again proceeds exactly as in the ordinary symplectic case; with the vector fields and forms being log or elliptic. 

\begin{lem} \textbf{(Compare Lemma 10.14 in \cite{McDuffSala1998})} \label{lem:exact}
Now assume that $M=T^*L(\log D_L)$ for some compact $(L,D_L)$ with either a logarithmic or elliptic structure. As previously discussed in the logarithmic case, the log or elliptic structure pulls back to the log/elliptic cotangent bundle, and there is a canonical log/elliptic symplectic form $\omega=\dx \lambda$, $\lambda$ the tautological log/elliptic one-form on the log/elliptic cotangent bundle. Let $\psi_t$ a log/elliptic symplectic isotopy on $M$, then 
\[ \Flux(\{\psi_t\}) = [\psi_1^*\lambda-\lambda] \] 
\end{lem} 
\begin{proof} 
Let $X_t$ be the family of log/elliptic symplectic vector fields which generates $\psi_t$. 
\begin{equation}\label{eq:exactpr} [i_{X_t}\omega]=[i_{X_t}\psi_t^*\omega] = [\psi_t^*(i_{(\psi_t)_*X_t}\omega)\stackrel{(1)}{=} [\psi_t^*(i_{X_t}\omega)]= [\psi_t^*(L_{X_t}\lambda)]\stackrel{(2)}{=}\frac{\dx}{\dx t}[\psi_t^*\lambda] \end{equation} 
For $(1)$, we use that for any (log/elliptic) symplectic vector field $X$ and any log/elliptic symplectomorphism $\phi$, $\phi_*X-X$ is Hamiltonian. $(2)$ is the application of the general identity for the Lie derivative of forms with respect to a time-dependent vector field: 
\[ \left.\frac{\dx }{\dx t'} \psi_{t'}^* \alpha\right\vert_{t'=t} = \psi_t^* L_{X_t} \alpha \] 
We obtain the result by integrating the right-hand side of (\ref{eq:exactpr}) from $t=0$ to $1$. 
\end{proof} 

\begin{cor} \label{cor:exact}
If $M=T^*L(\log D_L), D=T^*L(\log D_L)\vert_{D_L}$ with exact symplectic form, as above, 
\[\Flux(\pi_1(\Symp_0(M,\omega))) = 0.\] 
Thus we obtain a morphism $\Flux: \Symp_0(M,\omega) \rightarrow H^1(M,\log D)$
whose kernel is precisely given by Hamiltonian diffeomorphisms. 
\end{cor} 

\begin{proof} 
Lemma \ref{lem:exact} immediately proves the first sentence. This establishes that the flux only depends on the endpoint of a symplectic isotopy, and thus those with Hamiltonian endpoints are precisely those with zero flux. 
\end{proof} 

\subsection{Small deformations of Lagrangians in log symplectic manifolds} \label{sec:logdef}
In this section, we consider compact Lagrangians $(L,L\cap Z)$ inside a log symplectic manifold $(M,Z,\omega)$, and we assume that $L\pitchfork Z$. We have already proved a Lagrangian neighbourhood theorem for this scenario, Theorem \ref{thm:lognbhd}. We consider a neighbourhood $U\cong T^*L(\log L\cap Z)$ of $L$ as in this theorem, equipped with the canonical log symplectic form $\omega= \dx \lambda$. 

\begin{defn} 
A \emph{strong map of pairs} $f:(A,B)\rightarrow (M,N)$ (where $B\subset A, N\subset M$ are smooth submanifolds) is a smooth map with $f^{-1}(N)=B$. 
\end{defn} 

\begin{defn}\label{def:logdef}
A \emph{small deformation} of $L$ is a second Lagrangian $L'$ connected to $L$ by a smooth family of Lagrangians 
\[ \phi: L\times [0,1] \rightarrow M, \phi(L,0)=0, \phi(L,1)=L', \]
where each $\phi_t:=\phi(\cdot,t): (L,L\cap Z) \rightarrow (M,Z)$ is a smooth embedding and a strong map of pairs, and all $\phi_t$ are $C^1$-close to $\phi_0$. 
\end{defn} 
Note that if we have any smooth family of log Lagrangian embeddings $\phi: L\times [0,1] \rightarrow M, \phi(L,0)=L$, for sufficiently small $t$, $\phi_t$ will always be $C^1$-close to $\phi_0$. Furthermore, the images $\phi(L,t)$ will intersect the fibres in the tubular neighbourhood $U\cong T^*L(\log L\cap Z)$ of $L$ transversely. \\
Any such small deformation of $L$ can then be written as the graph of a logarithmic one-form on $L$ in $U$ -- and it is easy to check that with the canonical log symplectic form on $T^*L(\log L\cap Z)$, such a graph will be Lagrangian if and only if the log one-form is closed. 

\begin{prop} \label{prop:logdef}
Small deformations of $L$ up to local Hamiltonian isotopy (i.e. such that the image of $L$ never leaves $U$) are given by 
\[ H^1(L,\log L\cap Z). \] 
\end{prop} 

\begin{proof} Given the results on the log flux homomorphism, this proof proceeds exactly as for ordinary Lagrangians in an ordinary symplectic manifold.   
If we consider the graph of a small one-form $\alpha$ on $L$ in $U$, this is connected to $L$ by the smooth family of Lagrangians 
\[\phi:L\times [0,1]\rightarrow U, \phi(l,t)= \operatorname{Graph}(t\alpha)\vert_l, l\in L \]
If $\alpha=\dx f$ is exact, the Hamiltonian isotopy $\psi_t$ given by the flow of the Hamiltonian vector field $X_f$ maps $\psi_1: L\mapsto \operatorname{Graph}(\dx f)$.  

What is left to show: If the graph of a closed log one-form $\alpha\in \Gamma(T^*L(\log L\cap Z))$ is the image of $L$ under a Hamiltonian isotopy, $\alpha$ must have been exact. \\ 
We assume that there is a Hamiltonian isotopy $\{\psi_t\}$ such that $\psi_1(L)=\operatorname{Graph}(\alpha)$. We know that $\Flux(\{\psi_t\})=0$. According to Lemma \ref{lem:exact}, we have 
\[ \Flux(\{\psi_t\})= [\psi_t^*\lambda - \lambda] \Rightarrow \psi_1^*\lambda=\lambda + \dx g, g\in C^{\infty}(U).\]
Denote by $\iota: L \rightarrow T^*L(\log L\cap Z)$ the inclusion of $L$ as the zero section and view the section $\alpha$ as a map $\alpha: L \rightarrow T^*L(\log L\cap Z)$.
We have $\psi_1(L)=\operatorname{Graph}(\alpha)$, so $\alpha=\psi_1 \circ \iota$. \\
Now, the canonical one-form $\lambda$ on $T^*L(\log L\cap Z)$ has the property $\sigma^*\lambda=\sigma\ \forall \sigma \in \Omega^1(L,\log L\cap Z)$. So: $ \alpha = \alpha^*\lambda = \iota^*(\psi_1^*(\lambda))=\iota^*(\lambda + \dx g)= \dx \iota^*(g) $.
\end{proof} 

\begin{rmk} 
We can immediately make an analogous statement for compact Lagrangian branes in a stable generalized complex manifold which intersect the degeneracy locus $D$ transversely: For such branes, there is the analogous Lagrangian neighbourhood theorem by \cite{Cavalcanti2015a}, Theorem \ref{thm:lnbhd1}, and we can then apply exactly the same reasoning as above to conclude that small deformations of such a brane $(L,L\cap D)$ up to Hamiltonian isotopy are given by the first elliptic cohomology 
\[ H^1(L,\log\abs{L\cap D}). \] 
\end{rmk} 

\begin{rmk}
Lagrangians in logarithmic symplectic manifolds and generalized complex branes in stable generalized complex manifolds which intersect the degeneracy locus transversely are both examples of coisotropic submanifolds with respect to the associated Poisson structures. 
\cite{Cattaneo2005} investigate an $L_{\infty}$-algebroid structure on the normal bundle $NC$ of each coisotropic submanifold $C$, and \cite{Schatz2013} show that if the Poisson structure is \emph{fibrewise entire} on a tubular neighbourhood of $C$, small deformations of $C$ as a coisotropic submanifold are given by the degree-1 Maurer-Cartan elements of the $L_{\infty}$-algebroid. 
In both cases discussed here, the Lagrangian neighbourhood theorems show that $\omega^{-1}$ is indeed fibrewise entire, and we can show that the degree-1 Maurer-Cartan elements of the $L_{\infty}$-structure are just $\Omega^1_{\operatorname{cl}}(L,\log \partial L)$ or $\Omega^1_{\operatorname{cl}}(L,\log\abs{L\cap D})$ respectively. Thus these results from Poisson geometry provide an alternative way to obtain these deformation results. 
\end{rmk} 

\section{Real oriented blow-up of stable generalized complex manifolds} \label{sec:blow}
In this section we establish a connection between elliptic symplectic geometry and log symplectic geometry via the so-called \emph{real oriented blow-up}. In particular this allows us to relate every stable generalized complex manifold to a log symplectic manifold, and the Lagrangian branes to Lagrangian submanifolds with respect to the log symplectic structure. 

The real oriented blow-up of a compact submanifold $Y$ involves replacing this submanifold with its normal sphere bundle $S^1(NY)$, thus producing a manifold with boundary $[M;Y]$, a procedure that is well-defined up to diffeomorphism.\footnote{Recall that in a stable generalized complex manifold $(M,D)$ the radial residue of the elliptic symplectic form naturally lives on $S^1(ND)$, the normal sphere bundle to $D$.}
There is always a natural lift of vector fields on $M$ tangent to $Y$ to vector fields on $[M;Y]$ tangent to $\partial[M;Y]=S^1(NY)$. 

\begin{prop} \label{prop:blow}
\begin{enumerate}[label=(\roman*)]
\item If $E\rightarrow Y$ is a complex line bundle, its complex vector bundle structure induces a smooth free $U(1)$-action on the real oriented blow-up $\tilde{E}=[E;Y]$, which restricts to $\tilde{Y}:=\partial \tilde{E}$ s.t. $\beta\vert_{\tilde{Y}}:\tilde{Y}\rightarrow Y$ is a principal $U(1)$-bundle. 
\item If $\tilde{M}$ is a manifold with boundary $\tilde{D}$ s.t. there is a $U(1)$-principal bundle structure $\beta:\tilde{D}\rightarrow D$, there exists a smooth structure without boundary on 
\[ M:= (\tilde{M}\setminus \tilde{D})\sqcup D, \] 
which is canonical up to diffeomorphism and such that $[M;D]\cong \tilde{M}$. \\
$\beta$ extends to a $\beta: \tilde{M} \rightarrow M$ and $D$ has a complex normal bundle in $M$. \\ 
\item Let $(\tilde{M},\tilde{D})$ be a manifold with boundary $\partial\tilde{M}=\tilde{D}$ such that $\beta: \tilde{D}\rightarrow D$ is a principal $U(1)$-bundle. 
Then there is a canonical induced complex divisor on $(M,D)$ which is determined up to isomorphism. 
\end{enumerate}
\end{prop} 
\begin{proof} \begin{enumerate}[label=(\roman*)]
\item Since $E \rightarrow Y$ is a complex line bundle, its structure group can be reduced to $U(1)$. Its sphere bundle $S^1(E)$ is in fact the associated principal $U(1)$-bundle associated to $E$, whose $U(1)$-action naturally extends to $[E;Y]\cong S^1(E)\times \Rnum_{\geq 0}$ as a free action. 
\item Consider the open cover of $\tilde{M}$ given by $\tilde{M}\setminus \tilde{D}$ and a collar neighbourhood of $\tilde{D}$, $\tilde{U}\cong \tilde{D}\times [0,1)$. Since $\tilde{D}$ has the structure of a $U(1)$-principal bundle, it comes with a free $U(1)$-action that we can extend to the collar neighbourhood (this depends on the choice of collar neighbourhood, of course).
 Collapse the $U(1)$-fibres in $\tilde{U}$ to obtain the quotient space $U$, which still carries a $U(1)$-action, now with fixed-point set $D$. $U$ can be viewed as a neighbourhood of the zero section of the complex vector bundle associated to $\beta:\tilde{D}\rightarrow D$. 
Glue $U$ to $\tilde{M}\setminus \tilde{D}$ using the same gluing functions as for $\tilde{U}$. This is $M$. From the description of the real oriented blow-up, it is clear that $\tilde{M}\cong [M;D]$. It is not hard to see that this is, up to diffeomorphism, independent of the choice of tubular neighbourhood. 
\item We have already established that the $U(1)$-action on $\tilde{D}$ induces a complex structure on $ND$, the normal bundle to $D$ in $M$. Pick a hermitian metric $h$ for this complex line bundle. Then the function 
\[ ND \rightarrow \Rnum, v \mapsto h(v,v) \] 
clearly defines an elliptic divisor for $D$ in $\operatorname{tot}(ND)$, and since we have already picked an orientation of $ND$, this induces a complex divisor on $\operatorname{tot}(ND)$ (the complex line bundle being $p^*ND\rightarrow \operatorname{tot}(ND)$, where $p:ND\rightarrow D$ is the vector bundle projection). 
$h$ is unique up to multiplication with a positive real function on $D$, but such a rescaling is merely an isomorphism of elliptic divisors. Thus the induced complex divisor is also unique up to isomorphism. 

We can embed a neighbourhood of the zero section in $ND$ into $M$ as a tubular neighbourhood of $D$, and extend the elliptic divisor defined by $h$ to the entirety of $M$ in some positive smooth way. Since the real line bundle associated to an elliptic divisor is trivial, this is again unique up to isomorphism. 
\end{enumerate} 
\end{proof} 
Note that the choice of a defining function for $\tilde{D}$ in $\tilde{M}$ and an extension of the $U(1)$-action to a collar neighbourhood of $\tilde{D}$  fix the choice of elliptic divisor on a tubular neighbourhood of $D$, if these are to be compatible with the blow-down map $\beta: \tilde{M}\rightarrow M$. 

\begin{figure}[h] \centering
\includegraphics[scale=0.7]{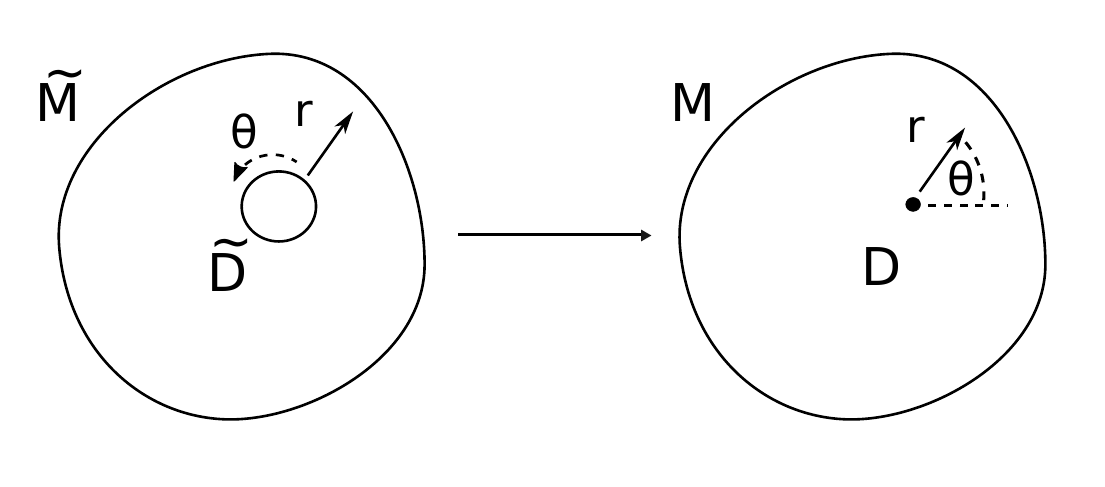}
\caption{Real oriented blow-up of the anticanonical divisor}
\end{figure} 

\begin{thm}\textbf{(Relation between stable generalized complex and log symplectic manifolds via real oriented blow-up I)} \label{thm:blowup11}\footnote{Note that this relation was first proposed by M. Gualtieri at Poisson 2010. The author's thesis is the first published instance of formal treatment and further development of the theory.}
Every stable generalized complex manifold $(M,D,\omega)$ can be related to a log symplectic manifold with boundary $(\tilde{M},\tilde{D},\tilde{\omega})$ via the real oriented blow-up of the anticanonical divisor. With $\beta: \tilde{M}\rightarrow M$ the blow-down map, 
\[ \tilde{\omega}=\beta^*(\omega),\]
which is a non-degenerate log two-form.  
\end{thm} 

This result is a consequence of the following lemma:
\begin{lem} Let $D=(U,s)$ be a complex divisor on a smooth manifold $M$; also denote its zero locus by $D$. Consider the real oriented blow-up $\tilde{M}:=[M;D],\tilde{D}:=\partial\tilde{M}$. 
\begin{enumerate}[label=(\roman*)]
\item The lift of vector fields associated with the real oriented blow-up induces a map 
\[ \beta^*: \Gamma(TM(-\log \abs{D})) \rightarrow \Gamma (T\tilde{M}(-\log \tilde{D})) \] 
which maps a local $C^{\infty}(M)$-basis to a local $C^{\infty}(\tilde{M})$-basis. 
\item 
There is a well-defined vector bundle morphism 
\[ \beta_*:T\tilde{M}(-\log \tilde{D})\rightarrow TM(-\log\abs{D}). \] 
\item This $\beta_*: T\tilde{M}(-\log \tilde{D})\rightarrow TM(-\log\abs{D})$ further induces a pullback 
\[ \beta^*: \Omega^{\bullet}(M,\log \abs{D}) \rightarrow \Omega^{\bullet}({\tilde{M}},\log \tilde{D}), \] 
which maps local bases to local bases,
and a vector bundle morphism 
\[ \beta_*: \wedge^k T^*\tilde{M}(\log \tilde{D}) \rightarrow \wedge^k T^*M(\log \abs{D}). \] 
\item The pullback $\beta^*$ induces an isomorphism on cohomology: 
\[ \beta^*: H^{\bullet}(M,\log\abs{D})\rightarrow H^{\bullet}(\tilde{M},\log\tilde{D}). \] 
\end{enumerate} 
\end{lem} 
\begin{proof}
\begin{enumerate}[label=(\roman*)]
\item All vector fields which are tangent to $D$ lift to logarithmic vector fields w.r.t $\tilde{D}$ on the real oriented blow-up $\tilde{M}$, and elliptic vector fields are certainly tangent to $D$. 
On $\tilde{M}\setminus \tilde{D}$, $\beta$ is a diffeomorphism, so local bases of vector fields get mapped to each other. Coordinates $(y^i,\theta,r)$ on a neighbourhood of $D$ pull back to coordinates on a neighbourhood of $\tilde{D}$. The associated local basis of elliptic vector fields is 
\[ \left(r\party{}{r},\party{}{\theta},\party{}{y^i}\right),\] 
which clearly lift to a local basis of log vector fields around $\tilde{D}$. 
\item Since $\beta\vert_{\tilde{M}\setminus\tilde{D}}$ is a diffeomorphism, and $T\tilde{M}(-\log\tilde{D}), TM(-\log \abs{D})$ are isomorphic to $T\tilde{M},TM$ outside $\tilde{D},D$ respectively, it suffices to define $\beta_*:T\tilde{M}(-\log\tilde{D})\rightarrow TM(-\log\abs{D})$ on $\tilde{D}$ and ensure that it forms a smooth vector bundle isomorphism together with the standard definition away from $\tilde{D}$. \\ 
Choose a tubular neighbourhood of $D$; write $M=ND$, and $p: ND\rightarrow D$ for the projection of the normal bundle. Fix a hermitian metric for $ND$. Also choose a $U(1)$-principal connection $\alpha$ for $\beta\vert_{\tilde{D}}: \tilde{D}\rightarrow D$. 
This immediately defines a complex linear connection on the associated complex line bundle $ND$, and splittings of the short exact sequences of vector bundles 
\begin{align} 
0 \rightarrow \mathfrak{t}' \rightarrow &T\tilde{D} \rightarrow \beta^*(TD) \rightarrow 0 \label{equ:conn1} \\ 
0 \rightarrow VD \rightarrow &TM \rightarrow p^*(TD) \rightarrow 0 \\ 
0 \rightarrow \mathbf{R}\oplus \mathfrak{t} \rightarrow &TM(-\log \abs{D})\vert_D \rightarrow TD \rightarrow 0 \label{equ:conn3}
\end{align} 
The line bundle $\mathfrak{t}'$ is the trivial line bundle spanned by the $U(1)$-action vector field on $\tilde{D}$, $\party{}{\theta}$. Similarly, $\mathfrak{t}$ is the trivial line bundle spanned by the elliptic vector field $\party{}{\theta}\vert_D$ on $D$, and $\mathbf{R}$ is the trivial line bundle given by the elliptic vector field $r\party{}{r}$ restricted to $D$. 

The splitting of the third sequence (\ref{equ:conn3}) is induced by the splitting of the second: The connection is unitary, so the associated horizontal lift will lift any section of $p^*(TD)$ to an elliptic vector field. For any $X\in T_dD$ we can choose an extension to a section in $\Gamma(TD)$, which in turn lifts to $\bar{X}\in \Gamma(TM(-\log\abs{D})\vert_D)$. The splitting of the sequence is defined by 
\[X \mapsto \bar{X}\vert_d, \] 
which is independent of the chosen extension. Thus we obtain an isomorphism 
\[ TM(-\log \abs{D})\vert_D \cong \mathbf{R}\oplus \mathfrak{t}\oplus TD \] 

Upon choice of hermitian metric on $ND$, the real oriented blow-up of $ND$ along the zero section $[ND;D]$ is canonically identified with $\tilde{D}\times [0,\infty)$ ($\tilde{D}=S^1ND$). $T\tilde{M}(-\log \tilde{D})\vert_{\tilde{D}}$ is canonically isomorphic to $\tilde{\mathbf{R}}\oplus T\tilde{D}$, with $\tilde{\mathbf{R}}$ the trivial line bundle spanned by $r\party{}{r}\vert_{\tilde{D}}$. The $U(1)$-principal connection for $\tilde{D}$ with its associated splitting (\ref{equ:conn1}) then defines an isomorphism 
\[ T\tilde{M}(-\log\tilde{D})\vert_{\tilde{D}} \cong \tilde{\mathbf{R}} \oplus \mathfrak{t}' \oplus \beta^*(TD) \] 

There is now clearly a vector bundle morphism $\tilde{\mathbf{R}}\oplus \mathfrak{t}'\oplus \beta^*(TD) \rightarrow \mathbf{R}\oplus \mathfrak{t}\oplus TD$ over $\beta\vert_{\tilde{D}}$ given by 
\begin{align*} 
\left.r\party{}{r}\right\vert_{\tilde{d}} &\mapsto \left.r\party{}{r}\right\vert_{\beta(\tilde{d})} \\ 
\left.\party{}{\theta}\right\vert_{\tilde{d}} &\mapsto \left.\party{}{\theta}\right\vert_{\beta(\tilde{d})} \\ 
(X,\tilde{d}) \in \beta^*(TD)_{\tilde{d}} &\mapsto X \in T_{\beta(\tilde{d})}D 
\end{align*} 
which in turn allows us to define 

\begin{minipage}{\linewidth} \centering
\begin{tikzpicture}[scale=0.75]
\draw (0,0)node{$T\tilde{M}(-\log\tilde{D})\vert_{\tilde{D}}$} (5,0)node{$\tilde{\mathbf{R}}\oplus \mathfrak{t}'\oplus \beta^*(TD)$} ; 
\draw (0,-2)node{$TM(-\log D)\vert_D$} (5,-2)node{$\mathbf{R}\oplus \mathfrak{t} \oplus TD$} ; 
\draw[->] (1.7,0)--(3.2,0); \draw[->] (5,-0.5)--(5,-1.5); 
\draw[->] (0,-0.5)--(0,-1.5); \draw[->] (1.7,-2)--(3.5,-2) ; 
\draw (-0.5,-1)node{$\beta_*$} ;
\draw(2.5,0.3)node{$\cong$} (2.5,-1.7)node{$\cong$};
\end{tikzpicture} 
\end{minipage} 
This joins smoothly with the standard definition of $(\beta\vert_{\tilde{M}\setminus\tilde{D}})_*: \tilde{M}\setminus \tilde{D} \rightarrow M\setminus D$ to give 
\[ \beta_*: T\tilde{M}(-\log \tilde{D})\rightarrow TM(-\log\abs{D}) \] 
The way we have defined this map, the definition of $\beta$ a priori depends on the choice of $U(1)$-principal connection. But since we have used the same connection to define the isomorphism in both the upper and the lower line of the above diagram, $\beta_*$ is invariant under a change of this connection. 

\item These morphisms can be defined using the ones in (i) and (ii): If $\alpha\in \Omega^k_M(-\log \abs{D})$ and $\tilde{X}_1,\dots, \tilde{X}_k \in T\tilde{M}(-\log \tilde{D})$, 
\[ \beta^*(\alpha)(\tilde{X}_1,\dots,\tilde{X}_k):= \alpha(\beta_*(\tilde{X}_1),\dots,\beta_*(\tilde{X}_k)) \] 
is a well-defined map of differential complexes. It is easy to check that local bases get mapped to local bases using the usual coordinates around $D$, $(y^i,\theta,r)$.  \\ 
Similarly, if $\tilde{\alpha} \in \wedge^k T^*M(\log \tilde{D})$, we can define a pointwise pushforward using the lift of elliptic to log vector fields: Let $X_1,\dots,X_k\in \Gamma(TM(-\log\abs{D}))$.
\[ \beta_*(\alpha)(X_1,\dots,X_k):= \alpha(\beta^*(X_1),\dots,\beta^*(X_k)). \] 
\item Recall that the following results for the logarithmic and elliptic de Rham cohomology are known: 
\begin{align*} 
H^k(M,\log\abs{D})&\cong H^k(M\setminus D) \oplus H^{k-1}(S^1 ND), [\alpha]\mapsto ([\alpha\vert_{M\setminus D}], [\res_r \alpha]) \\ 
H^k(\tilde{M},\log \tilde{D}) &\cong H^k(\tilde{M})\oplus H^{k-1}(\tilde{D}), [\tilde{\alpha}] \mapsto ([\tilde{\alpha}-s(\res\tilde{\alpha})],[\res \tilde{\alpha}]), \\ 
\text{where }s: H^{k-1}(\tilde{D})&\rightarrow H^k(\tilde{M},\log \tilde{D}) \text{ is some section of } \res 
\end{align*} 
$\tilde{M}$ is a manifold with boundary $\partial\tilde{M}=\tilde{D}$, and so $H^k(\tilde{M})\cong H^k(\tilde{M}\setminus \tilde{D})\cong H^k(M\setminus D)$.

\[ H^k(\tilde{M},\log \tilde{D}) \rightarrow H^k(\tilde{M}\setminus \tilde{D})\oplus H^{k-1}(\tilde{D}), [\tilde{\alpha}]\mapsto ([\tilde{\alpha}\vert_{\tilde{M}\setminus \tilde{D}}],[\res \tilde{\alpha}]) \] 
is an isomorphism, since we can map 
\begin{align*} [\tilde{\alpha}] &\mapsto ([\tilde{\alpha}-s(\res\tilde{\alpha})],[\res \tilde{\alpha}]) \\
&\mapsto \left([\tilde{\alpha}-s(\res\tilde{\alpha})]\vert_{\tilde{M}\setminus\tilde{D}},[\res\tilde{\alpha}]\right) \\ 
&\mapsto \left([\tilde{\alpha}]\vert_{\tilde{M}\setminus\tilde{D}},[\res\tilde{\alpha}]\right) 
\end{align*} 
Each of these maps is an isomorphism on cohomology, thus the composition is. 

Now, we obtain a commutative diagram 

\begin{minipage}{\linewidth} \centering
\begin{tikzpicture}[scale=0.8]
\draw[->] (1.5,0)--(3.5,0); \draw[->] (5,-0.5)--(5,-1.5) ; 
\draw[->] (0,-0.5)--(0,-1.5); \draw[->] (2.4,-2)--(3.2,-2); 
\draw (0,0)node{$H^k(M,\log\abs{D})$} (5,0)node{$H^k(\tilde{M},\log \tilde{D}$)} (-0.5,-2)node{$H^k(M\setminus D)\oplus H^{k-1}(S^1ND)$} (5.5,-2)node{$H^k(\tilde{M}\setminus \tilde{D})\oplus H^{k-1}(\tilde{D})$} ;
\draw (2.7,-1.5)node{Id} (2.7,0.5)node{$\beta^*$}; 
\draw (-0.5,-1)node{$\cong$} (5.5,-1)node{$\cong$};
\end{tikzpicture} 
\end{minipage} 

Three of the morphisms in this diagram are isomorphisms, so $\beta^*$ is one as well. 
\end{enumerate}
\end{proof} 
Theorem \ref{thm:blowup11} follows immediately from this result. 

The following theorem illustrates the converse relationship: 

\begin{thm} \label{thm:blowdown}\textbf{(Relation between stable generalized complex and log symplectic manifolds via real oriented blow-up II)}
Let $(\tilde{M},\tilde{D}=\partial\tilde{M},\tilde{\omega})$ be a real logarithmic symplectic manifold with a $U(1)$-principal bundle structure $\beta: \tilde{D} \rightarrow D$ and associated blow-down $\beta: (\tilde{M},\tilde{D})\rightarrow (M,D)$ (see Proposition \ref{prop:blow}, (ii)). Assume that 
\begin{enumerate} 
\item $i_{\party{}{\theta}} (\res \tilde{\omega}) = 0$, where $\party{}{\theta}$ is the action vector field of the $U(1)$-action on $\tilde{D}$. 
\item $\dx \left( i_{\party{}{\theta}} \tilde{\omega}\middle\vert_{\tilde{D}} \right) = 0$. Note that the first assumption implies that $i_{\party{}{\theta}}\tilde{\omega}\vert_{\tilde{D}}$ is actually a smooth one-form on $\tilde{D}$, so its exterior derivative on $\tilde{D}$ is defined. 
\end{enumerate} 
Then $(M,D)$ carries and induced elliptic divisor and $\tilde{\omega}$ induces a gauge equivalence class of stable generalized complex structures $\omega$ with anticanonical divisor $D$. 
\end{thm} 
\begin{proof} 
Since $\beta: \tilde{D}\rightarrow D$ carries the structure of a $U(1)$-principal bundle, $ND$ is a complex vector bundle and thus in particular oriented. 
Let $\tilde{\pi}=\tilde{\omega}^{-1}$ be the Poisson structure associated to the real logarithmic symplectic structure on $\tilde{M}$. $\tilde{\pi}^n \in \Gamma(\wedge^{2n}T\tilde{M})$ is then a section that vanishes transversely on $\tilde{D}$. \\ 
\textbf{Claim:} $\beta_*\tilde{\pi}^n \in \Gamma(\wedge^{2n}TM)$ defines an elliptic divisor with vanishing locus $D\subset M$. 
Clearly, the pointwise pushforward of $\tilde{\pi}^n$ is defined, we only need to ensure that this actually gives a smooth section, and that it has a positive definite normal Hessian on $D$. 

According to the normal form theorem for log symplectic forms, we can choose a collar neighbourhood around $\tilde{D}$ with coordinate $r$ such that the log symplectic form $\tilde{\omega}$ takes the form 
\[ \tilde{\omega} = \frac{\dx r}{r} \wedge \tilde{\Omega}_I + \Sigma, \] 
where $\Sigma$ is a closed two-form on $\tilde{D}$ and $\tilde{\Omega}_I= \res \tilde{\omega}$ is a closed one-form on $\tilde{D}$ (which are pulled back to the collar neighbourhood). \\ 
By assumption, $i_{\party{}{\theta}} (\res \tilde{\omega}) = 0$, so 
\[ L_{\party{}{\theta}} \tilde{\Omega}_I = 0 \] 
Furthermore, $i_{\party{}{\theta}} (\res \tilde{\omega}) = 0$ implies that $i_{\party{}{\theta}}\tilde{\omega}\vert_{\tilde{D}}=i_{\party{}{\theta}}\Sigma=: \tilde{\Omega}_R$. By assumption, this form is closed, and $L_{\party{}{\theta}}\tilde{\Omega}_R= 0$. 

Taken together, we obtain that $\tilde{\Omega}_I,\tilde{\Omega}_R$ are \emph{horizontal} one-forms on $\tilde{D}$, which are invariant under the $U(1)$-action. This implies that they are pulled back from smooth (closed) one-forms $\Omega_I,\Omega_R$ on $D$. 

The next step is to show that given the above, the $U(1)$-action on $\tilde{D}$ can always be extended to a collar neighbourhood $\tilde{U}=\tilde{D}\times [0,1)$ in such a way that $L_{\party{}{\theta}} \tilde{\omega}=0$ on all of $\tilde{U}$. We begin by extending the $U(1)$-action to $\tilde{U}$ as a free action in some way. (This amounts to fixing the diffeomorphism $\tilde{D}\times [0,1) \cong \tilde{U}$.) Assume $L_{\party{}{\theta}}\tilde{\omega} \neq 0$. Consider the family of logarithmic forms 
\[ \tilde{\omega}_t:= (e^{it})^*\tilde{\omega}, t \in [0,2\pi). \] 
Here $(e^{it})^*$ is the pullback with respect to the $U(1)$-action diffeomorphism. \\ 
We can average over $t$ to obtain the $U(1)$-invariant log form 
\[ \bar{\omega} = \frac{1}{2\pi}\int_{S^1} \tilde{\omega}_t \dx t \] 
We have
\[\left. \frac{\dx}{\dx t}\right\vert_{t=0} \tilde{\omega}_t = L_{\party{}{\theta}} \tilde{\omega} = \dx\left(i_{\party{}{\theta}} \tilde{\omega}\right) \] 
Since we assumed $\dx\left(i_{\party{}{\theta}} \tilde{\omega}\vert_{\tilde{D}}\right) = 0$, we obtain $\tilde{\omega}_t\vert_{\tilde{D}}=\tilde{\omega}\vert_{\tilde{D}}=\bar{\omega}\vert_{\tilde{D}}$. In particular, $\bar{\omega}$ is also non-degenerate, i.e. a log symplectic form, at least upon restriction to a smaller collar neighbourhood of $\tilde{D}$. 
The difference $\tilde{\omega}-\bar{\omega}$ is a smooth two-form which vanishes on $\tilde{D}$. The blown-up locus $\tilde{D}$ is a deformation retract of its collar neighbourhood, so $\tilde{\omega}-\bar{\omega}=\dx \alpha$, with $\alpha$ a smooth one-form on the collar neighbourhood of $\tilde{D}$. \\ 
Since $\tilde{D}$ was assumed to be compact and $\tilde{\omega}\vert_{\tilde{D}}=\bar{\omega}\vert_{\tilde{D}}$, we can apply the Moser argument to the family of non-degnerate (on a small neighbourhood of $\tilde{D}$) log forms 
\[ \omega'_s:= s \tilde{\omega} + (1-s)\bar{\omega}, s\in [0,1] \] 
The logarithmic vector field $X_s:= (\omega'_s)^{-1}(\alpha)$ integrates to an isotopy $\phi_s$ with 
\[ \phi_s^*(\omega'_s)=\bar{\omega},\ \phi_1^*(\tilde{\omega}) = \bar{\omega}. \] 
Write $\phi:=\phi_1$ for this diffeomorphism. We have: 
\[ 0=L_{\party{}{\theta}} \phi^*(\tilde{\omega})= \phi^*\left(L_{\phi_*\party{}{\theta}} \tilde{\omega}\right) \Rightarrow L_{\phi_*\party{}{\theta}}\tilde{\omega} = 0 \] 
Since $\phi$ is a diffeomorphism, $\phi_*\party{}{\theta}$ is again the action vector field of a $U(1)$-action on a collar neighbourhood of $\tilde{D}$, and this is the extension of the $U(1)$-action we have been looking for. 

Locally on an open set in $\tilde{D}$, we can now write 
\[ \Sigma = \dx \theta \wedge \tilde{\Omega}_R + \tilde{\sigma}, \text { where } i_{\party{}{\theta}} \tilde{\sigma} = 0, \] 
so on a neighbourhood near the boundary: 
\[ \tilde{\omega} = \frac{\dx r}{r} \wedge \tilde{\Omega}_I + \dx \theta \wedge \tilde{\Omega}_R + \tilde{\sigma},  \]
where $\theta$ is chosen such that $L_{\party{}{\theta}}\tilde{\omega}=0$ on the entire collar neighbourhood.  
$\tilde{\sigma}$ is a horizontal, closed two-form on $\tilde{D}$, i.e. it is also the pullback of a closed two-form $\sigma$ on $D$.  
\\
In particular, this implies that $\beta_*\tilde{\pi}^n$ is smooth on $D$. 
The normal Hessian of $\beta_*\tilde{\pi}^n$ as a section of $\wedge^{2n}TM$ is clearly positive definite, since $\tilde{\pi}^n$ itself vanishes transversely, and we obtain an elliptic divisor $(\beta_*\tilde{\pi}^n, \wedge^{2n} TM)$ with vanishing locus $D$, which is co-oriented. 

Using the local expression for $\tilde{\omega}$ on a collar neighbourhood established above, and the established fact that $\tilde{\Omega}_I=\beta^*\Omega_I, \tilde{\Omega}_R=\beta^*(\Omega_R),\tilde{\sigma}=\beta^*(\sigma)$ for smooth forms on $D$, it is clear that $\tilde{\omega}$ is the pullback of 
\[ \omega= \frac{\dx r}{r} \wedge \Omega_I + \dx \theta \wedge \Omega_R + \sigma. \]
Such an expression exists for each open set of a covering of $D$, and because the coordinate transformations for a tubular neighbourhood $U\cong ND$ are compatible with those for $\tilde{U}\cong \tilde{D}\times [0,1)$, these patch to a well-defined elliptic symplectic form $\omega \in \Omega^2(M,\log\abs{D})$.  

Lastly, we have 
\[ \res_e(\omega)= 0, \]  
so $\omega$ and the already established co-orientation of $D$ together define the gauge-equivalence class of a stable generalized complex structure on $(M,D)$. 
\end{proof} 

\subsection{Lagrangian branes under blow-up} 
Let $(\tilde{M},\tilde{D},\tilde{\omega})$ and $(M,D,\omega)$ be a logarithmic symplectic manifold and stable generalized complex manifold related by real oriented blow-up as above. 
We begin by showing that Lagrangian branes with and without boundary in $(M,D)$ lift to logarithmic Lagrangians in $(\tilde{M},\tilde{D})$: 
\begin{prop} \textbf{(Lift of Lagrangian branes to the real oriented blow-up)} 
\begin{enumerate}[label=(\roman*)]
\item Let $L\subset M$ be a Lagrangian brane which intersects $D$ transversely. Then $\tilde{L}=[L;L\cap D]\subset \tilde{M}$ is a Lagrangian submanifold with boundary which intersects $\tilde{D}$ transversely. 
\item Let $L\subset M$ be a Lagrangian brane with boundary. The \emph{lift} of $L$ is $\tilde{L}:= \operatorname{cl}\left(\beta^{-1}(L\setminus(L\cap D))\right)=\beta^*(L)$, where $\operatorname{cl}(U)$ denotes the closure of the subset $U$ in the ambient manifold. \\
Then the lift is a Lagrangian submanifold in $\tilde{M}$ which intersects the singular locus $\tilde{D}$ transversely, and such that $\party{}{\theta}$ is nowhere tangent to $\tilde{L}\cap \tilde{D}$. $\tilde{L}$ intersects each $U(1)$-fibre in at most one point. 
\end{enumerate}
\end{prop} 
\begin{proof} \begin{enumerate}[label=(\roman*)]
\item The blow-up $[L;L\cap D]$ naturally embeds into $(\tilde{M},\tilde{D})$ as $\tilde{L}=\operatorname{cl}\left(\beta^{-1}(L\setminus( L\cap D))\right)$: Since $L\pitchfork D$, we can always choose a small tubular neighbourhood of $D$ in which $L$ is a fibre of $ND$. 
The submanifold $\tilde{L}$ is clearly Lagrangian with respect to $\tilde{\omega}=\beta^*(\omega)$, since $\beta(\tilde{L})=L$. Obviously $\tilde{L}\pitchfork \tilde{D}$. 
\item Again, $\tilde{L}$ is clearly Lagrangian with respect to $\tilde{\omega}=\beta^*(\omega)$, since $\beta(\tilde{L})=L$. It intersects $\tilde{D}$ transversely: $\tilde{D}$ is codimension-1 in $\tilde{M}$. If $\tilde{L}$ did not intersect $\tilde{D}$ transversely, we would have $T\tilde{L}\vert_{\tilde{L}\cap\tilde{D}} \subset T\tilde{D}\vert_{\tilde{L}\cap\tilde{D}}$, which would also imply $TL\vert_{L\cap D} \subset TD\vert_{L\cap D}$, which is a contradiction -- we assumed $L\cap D$ to be a clean intersection. Because of the clean intersection, we can choose a tubular neighbourhood of $D$ such that $L$ is a rank-1 subbundle of $ND$ over $\partial L$. This means in particular that $\party{}{\theta}$ is not tangent to $L$ in some neighbourhood of $D$, so it will not be tangent to $\tilde{L}\cap\tilde{D}$ either. If $\tilde{L}$ intersected any $U(1)$-fibre in more than one point, $L\cap D$ would not be the boundary of $L$. 
\end{enumerate}\end{proof} 
\begin{figure}[h] \centering
\includegraphics[scale=0.4]{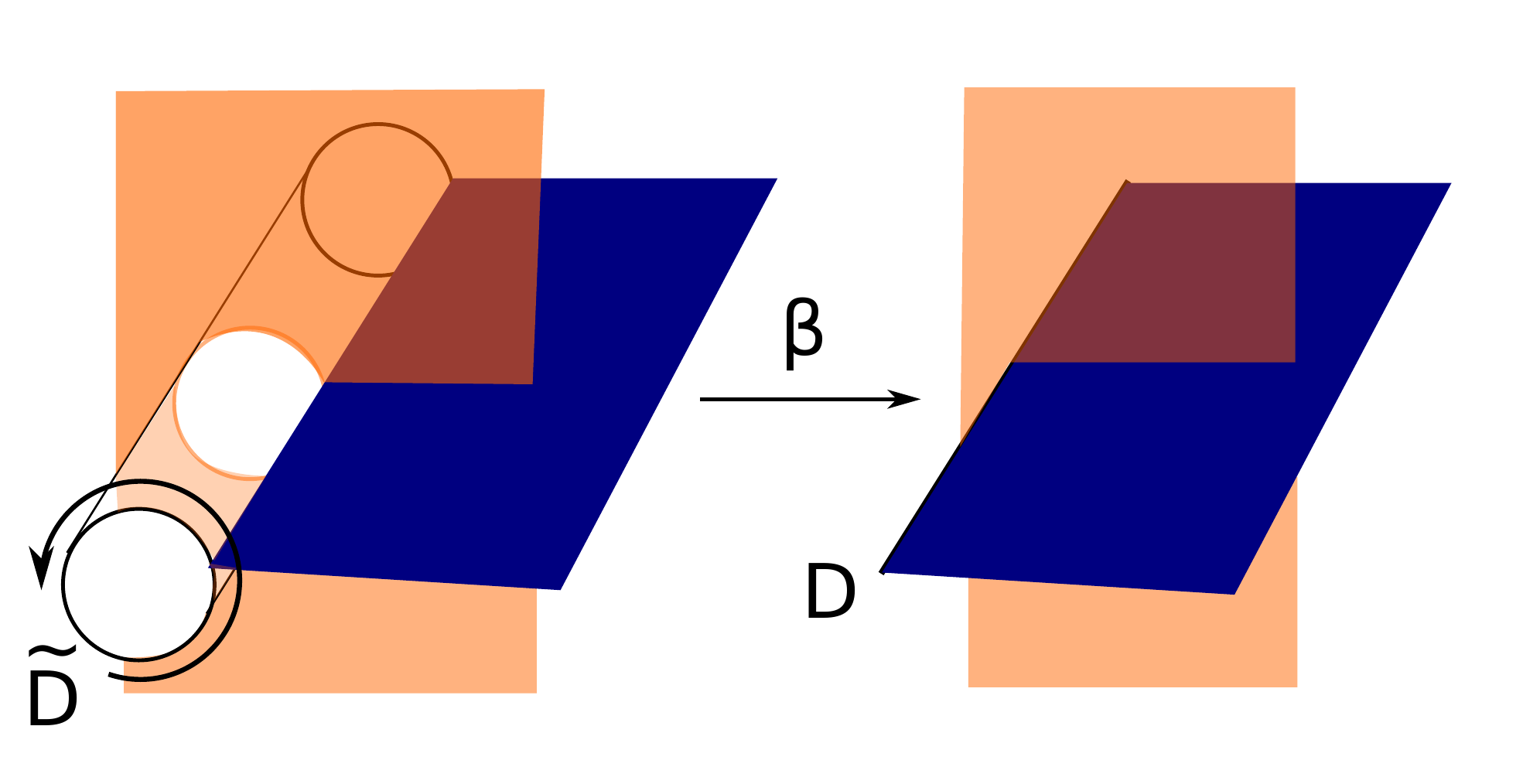}
\caption{Log Lagrangians under blow-down: Depending on the intersection with $\tilde{D}$, the result can be a brane with or without boundary.}
\end{figure}

Conversely,let $\tilde{L}\subset \tilde{M}$ be a compact Lagrangian submanifold with boundary, s.t. $\partial\tilde{L}\subset \tilde{D}$ and $\tilde{L}\pitchfork\tilde{D}$. According to Theorem \ref{thm:lognbhd} every connected component of $\partial\tilde{L}$ will lie inside a symplectic leaf of $\tilde{D}$ and be Lagrangian inside this leaf. We know that the symplectic foliation of $\tilde{D}$ is precisely given by the distribution $\ker(\res\tilde{\omega})=\ker(\tilde{\Omega}_I)$. 

There are two cases of interest, in which $\beta(\tilde{L})$ is a smooth submanifold in $M$: 

\begin{thm} \label{thm:branes}\textbf{(Blow-down of logarithmic Lagrangians)} 
Let $(M,D,\omega)$ be a stable generalized complex manifold and $(\tilde{M},\tilde{D},\tilde{\omega})$ its blow-up. Let $\tilde{L}\subset \tilde{M}$ be a Lagrangian submanifold with boundary that intersects $\tilde{D}$ transversely, and $\tilde{D}\cap\tilde{L}=\partial{L}$. 
\begin{enumerate}[label=(\roman*)]
\item If $\party{}{\theta}\in \Gamma(T(\tilde{L}\cap\tilde{D}))$, the $U(1)$-action restricts to $\partial {\tilde{L}}$. 
Any such $\tilde{L}$ is log Hamiltonian isotopic to a $\tilde{L}'$ whose image under the blow-down map $L'=\beta(\tilde{L}')$ is a smooth Lagrangian brane without boundary. 
\item If $\party{}{\theta}$ is nowhere tangent to $T(\tilde{L}\cap\tilde{D})$ and $\beta\vert_{\tilde{L}}$ is injective, $\beta(\tilde{L})=:L$ is a Lagrangian brane with boundary in $D$. 
\footnote{In Example \ref{ex:Hopf}, we consider a similar example where $\beta\vert_{\tilde{L}}$ is \emph{not} injective, but $\beta(\tilde{L})$ is nonetheless smooth and intersects $D$ in a smooth codimension-one submanifold, just not its boundary. But since branes with boundary are the focus of this text, this case is excluded for now.} 
\end{enumerate} 
\end{thm} 
\begin{proof} \begin{enumerate}[label=(\roman*)] 
\item In this case, the free $U(1)$ action restricts to $\partial \tilde{L}$, which itself becomes a $U(1)$-principal bundle over $\beta(\partial \tilde{L})=D_L$. \\
Pick a collar neighbourhood of $\tilde{D}$, $\tilde{U}=\tilde{D}\times [0,1)$, in such a way that 
\[ \omega = \frac{\dx r}{r} \wedge \tilde{\Omega}_I + \dx \theta \wedge \tilde{\Omega}_R + \tilde{\sigma}, \] 
and $L_{\party{}{\theta}}\omega = 0$. 
We know that 
\[\iota^*_{\partial\tilde{L}} \tilde{\Omega}_I=0, \iota^*_{\partial\tilde{L}} (\dx \theta \wedge \tilde{\Omega}_R + \tilde{\sigma}) = 0, \]
so $\tilde{L}':= \partial \tilde{L} \times [0,1)$ is a log Lagrangian. Obviously the $U(1)$-action also restricts to its boundary.
Clearly, $\beta(\tilde{L}')=: L'$ is a smooth elliptic Lagrangian without boundary in $M$, equipped with a pullback elliptic divisor $D_{L'}$, and $L'\pitchfork D$. Pick an elliptic Lagrangian neighbourhood for $L'$ according to Theorem \ref{thm:lnbhd1}. This corresponds to a log Lagrangian neighbourhood of $\tilde{L}'$ in $\tilde{M}$ (via pullback of coordinates). 
At least in some neighbourhood of $\tilde{D}$, $\tilde{L}$ is contained in the thus obtained log neighbourhood of $\tilde{L}'$ and can be written as the graph of a closed log one-form $\tilde{\alpha}\in\Omega^1(\tilde{L}',\log \partial \tilde{L}')$. 
Since the pullback $\beta*: \Omega^1(L',\log \abs{D_{L'}}) \rightarrow \Omega^1(\tilde{L'},\log \partial \tilde{L}')$ is an isomorphism on cohomology, $\tilde{\alpha}$ has to be in the same cohomology class as a form $\tilde{\alpha}'$ which is the pullback of a smooth elliptic form $\alpha'$ on $L'$. Thus the graph of $\tilde{\alpha}$, $\tilde{L}$, is locally log Hamiltonian isotopic to a Lagrangian which blows down to a smooth Lagrangian brane without boundary in $M$. 

This argument takes place inside a tubular neighbourhood of $\tilde{D}$, it is however possible to cut off any Hamiltonian function with a bump function, so the Hamiltonian isotopy above can be extended by the identity outside a neighbourhood of $\tilde{D}$. 
\item Since $\party{}{\theta}$ is nowhere tangent to $\tilde{L}$, $\partial \tilde{L}$ intersects each $U(1)$-fibre transversely. Since $\beta\vert_{\partial \tilde{L}}$ is injective, the intersection with each fibre is either empty or in exactly one point. Thus $\beta\vert_{\tilde{L}}$ is a diffeomorphism onto its image, and $L:=\beta(\tilde{L})$ is a smooth submanifold of $M$, with boundary in $D$. 
\end{enumerate} \end{proof} 

\begin{ex} 
Consider $M=T^*L(-\log \abs{Y})$, where $L$ is a compact $n$-dimensional manifold and $Y\subset L$ a codimension-2 submanifold given as the zero-locus of a complex divisor. Equip $M$ with the canonical elliptic symplectic form $\omega_0$. 

Then the real oriented blow-up $\tilde{M}$ of $T^*L(-\log\abs{Y})\vert_Y$ inside $M$ with the pullback-form $\beta^*\omega_0$ is isomorphic to $T^*\tilde{L}(-\log\partial\tilde{L})$, where $\tilde{L}:=[L;Y]=\operatorname{cl}(L\setminus Y)$ is the lift under the real oriented blow-up, equipped with the canonical logarithmic symplectic form $\tilde{\omega}_0$.
  
Conversely, if $\tilde{L}$ is an $n$-dimensional manifold with boundary $\partial\tilde{L}$, s.t. $\partial\tilde{L}$ is a $U(1)$-principal bundle, we can consider the blow-down of $(T^*\tilde{L}(-\log \partial\tilde{L}),\tilde{\omega}_0)$, and the result will be isomorphic to $(T^*L(-\log\abs{Y}),\omega_0)$, where $L$ is the blow-down $\tilde{L}\rightarrow L$ and $Y=\partial\tilde{L}/U(1)$. 
\end{ex} 

\subsection{Neighbourhoods of Lagrangian branes with boundary} \label{sec:nbhdbdy}
We have now established that every Lagrangian brane with boundary $(L,\partial L) \subset (M,D)$ is the blow-down of a log Lagrangian submanifold with boundary $\tilde{L}$ which intersects $\tilde{D}$ transversely. \\ 
For such Lagrangians $\tilde{L}$, we have proved a Lagrangian neighbourhood theorem, Theorem \ref{thm:lognbhd}. Choose a neighbourhood $\tilde{U}$ of $\tilde{L}$ according to this theorem. Its image under the blow-down $U:=\beta(\tilde{U})$ is a wedge neighbourhood in the sense of Definition \ref{defn:wedge}. 
\begin{figure}[h] \centering 
\includegraphics[scale=0.45]{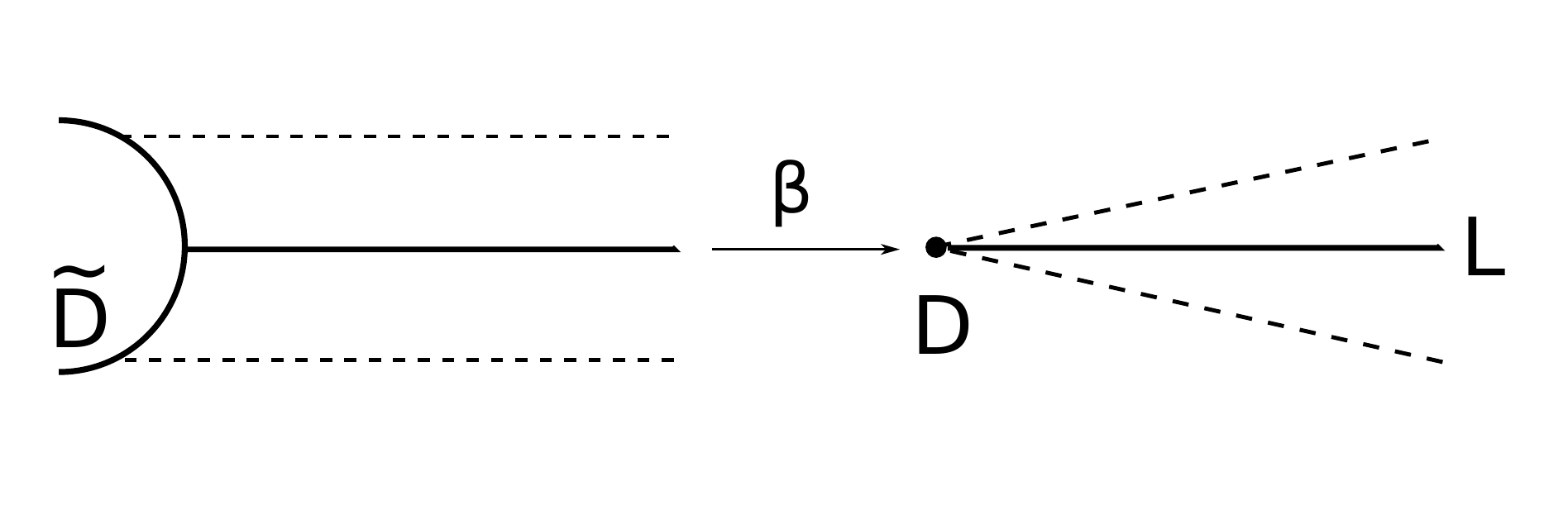}
\caption{Lagrangian wedge neighbourhood as a blow-down}
\end{figure}

\begin{prop} \label{prop:ellnbhd} \textbf{(Normal form for wedge neighbourhoods)}
Let $(L,\partial L)$ be a Lagrangian brane with boundary in a stable generalized complex manifold $(M,D,\omega)$, and $(\tilde{L},\partial \tilde{L})$ the corresponding log Lagrangian in the real oriented blow-up $(\tilde{M},\tilde{D},\tilde{\omega})$. Let $(\tilde{U},\partial \tilde{U})\subset T^*\tilde{L}(\log \partial\tilde{L})$ be a Lagrangian neighbourhood of $(\tilde{L},\partial \tilde{L})$ in the sense of Theorem \ref{thm:lognbhd}. Identify $\tilde{U}$ with the tubular neighbourhood of the zero section in $T^*\tilde{L}(\log \partial \tilde{L})$ and write $\tilde{U}$ for either. Recall that there is a diffeomorphism $\tilde{\psi}: \tilde{U}\rightarrow \tilde{U}$ such that \[\tilde{\psi}^*\left(\tilde{\omega}\vert_{\tilde{U}}\right)= \tilde{\omega}_0, \] 
where $\tilde{\omega}_0$ is the standard log symplectic form on the log cotangent bundle $T^*\tilde{L}(\log \partial \tilde{L})$. 

Now, there exists an elliptic symplectic form $\omega_0$ on the wedge neighbourhood
\[ (U,D_U):=(\beta(\tilde{U}),\beta(\partial \tilde{U})) \] 
which locally around each point in $D_U$ can be expressed in local coordinates $(r,x,q_i)$ on $L$ as 
\[ \omega_0= \frac{\dx r}{r} \wedge \dx y + \dx \theta \wedge \dx x + \sum_i \dx q^i \wedge \dx p_i. \] 
This form pulls back to $\tilde{\omega}_0$ on the blow-up of the wedge neighbourhood. 
Furthermore, the diffeomorphism $\tilde{\psi}$ of $\tilde{U}$ descends to a diffeomorphism of the wedge neighbourhood $\psi: U\rightarrow U$ such that $\psi^*\omega=\omega_0$. 
\end{prop}  
\begin{proof} 
First we need to show that $\tilde{\omega}_0$ can indeed be written as the pullback of an elliptic form on $U$ that extends smoothly around the wedge of $U$. 
According to Lemma \ref{lem:smooth}, we can choose $(r,\theta)$ in such a way that $r\party{}{r}$ is tangent to $L$ in a small neighbourhood of $D$, a property that persists after blow-up. 
We have already established that the brane in the blow-up, $\tilde{L}$, intersects each $U(1)$-fibre in at most one point, so we can choose the tubular neighbourhood of $\tilde{L}$ in such a way that near $\tilde{D}$, $\party{}{\theta}$ is tangent to the fibres. We can view $\party{}{\theta}\vert_{\tilde{L}}$ as spanning a sub-line bundle of $N\tilde{L}\vert_{\partial \tilde{L}}$.  

We identify $N\tilde{L}$ with $T^*\tilde{L}(\log \tilde{L}\cap\tilde{D})$ using $\tilde{\omega}$. 
On an open set near $\tilde{D}$, we can write 
\[ \tilde{\omega}= \frac{\dx r}{r} \wedge \gamma_r + \dx \theta \wedge \gamma_{\theta} + \tilde{\epsilon} .\] 
The subbundle of $N\tilde{L}$ spanned by $\party{}{\theta}$ gets mapped to the subbundle of $T^*L(\log \tilde{L}\cap\tilde{D})$ spanned by $\iota^*\gamma_{\theta}, \iota: \tilde{L}\hookrightarrow \tilde{M}$.  \\ 

Let $\xi_r$ be the fibre coordinate associated to $\frac{\dx r}{r}$, and $\chi$ the fibre coordinate associated to $\iota^*\gamma_{\theta}$. 
Then $\tilde{\omega}_0$ has the form 
\[ \tilde{\omega}_0= -\frac{\dx r}{r} \wedge \dx \xi_r - \iota^* \gamma_{\theta}\wedge \dx \chi + \tilde{\rho} \] 
with $\tilde{\rho}$ a two-form on $T^*\tilde{L}(\log \tilde{L}\cap\tilde{D})$ s.t. $i_{r\party{}{r}}\tilde{\rho}=0, i_{\party{}{\chi}}\tilde{\rho}=0$. 

This implies 
\[ i_{\party{}{\chi}}(\res \tilde{\omega}_0)=0, \dx\left( i_{\party{}{\chi}} \tilde{\omega}_0 \middle\vert_{\tilde{D}} \right) = \dx \left(\iota^*\gamma_{\theta}\vert_{\tilde{D}\cap\tilde{L}}\right)=0. \] 
 We can choose the tubular neighbourhood embedding for $\tilde{L}$ such that $\theta=\omega^*(\chi)$ (viewing $\omega$ as a map $N\tilde{L}\rightarrow T^*\tilde{L}(\log\tilde{L}\cap \tilde{D})$). 
According to Theorem \ref{thm:blowdown} this means that $\tilde{\omega}_0$ is indeed the pullback of a locally defined elliptic form $\omega_0$ on $(U,U\cap D)$, with respect to the same elliptic divisor as $\omega$: $(\tilde{\omega}^{-1})^n$ and $(\tilde{\omega}_0^{-1})^n$ are both of the form 
\[ f r\party{}{r}\wedge \party{}{\theta} \wedge \party{}{y_3} \wedge \dots \wedge \party{}{y_{2n}}, f\neq 0 \]
with respect to the same $(r,\theta)$. 

The diffeormorphism $\tilde{\psi}:\tilde{U}\rightarrow \tilde{U}$ relating $\tilde{\omega}$ and $\tilde{\omega}_0$ is the time-1 flow of the time-dependent log vector field 
\[ \tilde{X}_t = - \tilde{\omega}_t^{-1}(\tilde{\alpha}),\ \tilde{\omega}_t= t\tilde{\omega}+(1-t)\tilde{\omega}_0, \tilde{\omega}-\tilde{\omega}_0 = \dx \tilde{\alpha} \] 
We have: 
\[ \tilde{\omega}=\beta^*(\omega), \tilde{\omega}_0=\beta^*(\omega_0),\] 
$L$ is Lagrangian with respect to both $\omega$ and $\omega_0$, so $\omega-\omega_0=\dx \alpha$, where $\beta^*(\alpha)=\tilde{\alpha}$. 
Clearly, 
\[ \beta_*(\tilde{X}_t)=-\omega_t^{-1}(\alpha) =: X_t, \] 
which is an elliptic vector field whose time-1 flow takes $\psi^*(\omega)=\omega_0$. 
\end{proof} 
We thus obtain a standard local neighbourhood of branes with boundary in stable generalized complex manifolds, which is a wedge neighbourhood in the sense of Definition \ref{defn:wedge}. 

With the results from this section, 
it is easy to prove Proposition \ref{prop:resb}: 
\begin{proof}\textit{of Proposition \ref{prop:resb}} 
According to  Proposition \ref{prop:ellnbhd}, we can pick coordinates on a wedge neighbourhood $(U,U\cap D)$ of a brane with boundary so that 
\begin{align*} 
\omega &= \frac{\dx r}{r} \wedge \dx y + \dx \theta \wedge \dx x+ \sigma 
\end{align*} 
Then, up to addition of a smooth closed two-form, the real part of a corresponding log symplectic form $\sigma=B+i\omega$ is 
\[ B= \frac{\dx r}{r} \wedge \dx x + \dx \theta \wedge \dx y, \] 
with $\iota: L \rightarrow M$, $\res(\iota^*B) = \iota^*\dx x\vert_{\partial L}.$
In the proof above, we have already established that after real oriented blow-up 
\[ \tilde{\omega}: N\tilde{L}\stackrel{\cong}{\rightarrow} T^*\tilde{L}(\log \tilde{L}\cap \tilde{D}), \left.\party{}{\theta}\right\vert_L \mapsto \dx x \] 
The fact that this is an isomorphism ensures that $\iota^*\dx x\neq 0$ everywhere. 
\end{proof}

\section{Small deformations of Lagrangian branes with boundary in stable GC manifolds} \label{sec:defo}
Using the neighbourhood normal form result established in the previous section, we now consider small deformations of Lagrangian branes with boundary and established the local form of the smooth deformation space. 

\begin{defn} 
A \emph{small deformation} of the Lagrangian brane with boundary $L$ is a Lagrangian brane with boundary $L_1$ such that there exists a smooth family of embeddings
\[ \phi: (L,\partial L) \times [0,1] \rightarrow (M,D) \]
which are strong maps of pairs, such that $\phi(L,0)=L, \phi(L,1)=L_1$, and such that all $\phi(L,t)$ are Lagrangian branes with boundary in the sense of Definition \ref{def:bb} which are $C^1$-close to $L$. 
\end{defn} 
When we lift such a family $\phi(L,t)$ of Lagrangian branes with boundary to the real oriented blow-up of $(M,D)$, $(\tilde{M},\tilde{D})$, we obtain a family of log Lagrangians which define a $C^1$-small deformation of $\tilde{L}=\beta^*(L)$ in the sense of Definition \ref{def:logdef}. Conversely, the image of any sufficiently $C^1$-small deformation of $\tilde{L}$ under the blow-down map $\beta$ is a small deformation of Lagrangian branes with boundary as above. 

If we pick a log Lagrangian neighbourhood for $\tilde{L}$ and a $\beta$-related wedge neighbourhood of $L$, the lifts of sufficiently small deformations $\phi(L,t)$ to the blow-up will intersect the fibres of $T^*L(\log \partial L)$ transversely and can thus be written as graphs of closed log one-forms on $L$. Conversely, under the identification of the tubular neighbourhood of $\tilde{L}$ with a neighbourhood of the zero section of $T^*L(\log \partial L)$ followed by the blow-down, every sufficiently small closed log one-form defines a small deformation of the Lagrangian brane with boundary $L$. 

\begin{thm} \label{thm:defbdy}\textbf{(Small deformations of Lagrangian branes with boundary)} 
Up to local Hamiltonian isotopy (i.e. Hamiltonian isotopy that stays within the wedge neighbourhood) the small deformations of a brane with boundary $L\subset (M,D)$ are given by $H^1(L,\log L\cap D)$, i.e. the first cohomology of logarithmic forms on $L$ with respect to $\partial L=L\cap D$. 
\end{thm} 
\begin{rmk} 
Although wedge neighbourhoods are not tubular neighbourhoods in the usual sense, and drop in dimension by one over $\partial L$, these are the natural neighbourhoods containing small deformations of branes with boundary: According to the definition, we only consider deformations whose boundary stays inside $D$. All such small deformations precisely sweep out a wedge neighbourhood of $L$. 
\end{rmk}
In order to prove this theorem, we need the following lemma: 
\begin{lem} \label{lem:symham}
If a brane with boundary and a small deformation of it are related by a local log Hamiltonian isotopy in a tubular neighbourhood in $(\tilde{M},\tilde{D})$ and by a local elliptic symplectic isotopy in the corresponding wedge neighbourhood in $(M,D)$, there is also a local elliptic Hamiltonian isotopy between them in $(M,D)$. 
\end{lem} 
\begin{proof} 
Let $L,L'\subset (M,D)$ be branes with boundary that are related by an elliptic symplectic isotopy $\phi_t: \phi_1(L)=L'$ contained in a wedge neighbourhood $U$ of $L$, and $\tilde{L}'=\beta^*(L)$ is given as a section of $T^*L(\log \partial L)$ in the Lagrangian neighbourhood of $L$. 
Now, $\phi_t$ lifts to a log symplectic isotopy $\tilde{\phi}_t$ of $\tilde{L}=\beta^*(L), \tilde{L}'=\beta^*(L')$ by lifting the generating elliptic symplectic vector field. Denote 
\[ \operatorname{Flux}(\{\tilde{\phi}_t\})=[\alpha],\]  
where we can choose a representative $\alpha \in \Omega^1(\tilde{L},\log\tilde{L}\cap\tilde{D})$, viewed as a log form on $\tilde{U}$ via pullback, since $\tilde{U}=\beta^{-1}(U)$ is homotopy-equivalent to $\tilde{L}$. 
We further assumed that $\tilde{L},\tilde{L}'$ are related by a Hamiltonian isotopy $\psi_t$ in $(\tilde{M},\tilde{D})$. In particular, this will have vanishing flux, and we know that $\tilde{L}'$ is the graph of an exact log one-form on $\tilde{L}$. 
\[ (\psi_1)^{-1}(\tilde{L'})=\tilde{L}. \]  
Further, there is the standard symplectic isotopy $\Lambda_t^{-\alpha}$ between the zero section and the graph of $-\alpha$ with flux $[-\alpha]$ given by the symplectic vector field associated to $-\alpha$. 
The composition $\Lambda^{-\alpha}_t \circ \psi^{-1}_t \circ \tilde{\phi}_t$ has the properties 
\begin{align*}
(\Lambda^{-\alpha}_1\circ (\psi_1)^{-1} \circ \tilde{\phi}_1)(L)& =\operatorname{Graph}(-\alpha) \\ 
\operatorname{Flux}(\Lambda^{-\alpha}_t \circ (\psi_t)^{-1} \circ \tilde{\phi}_t) &= [\alpha-\alpha]=0 
\end{align*} 
Thus, $L$ and the graph of $-\alpha$ are related by a local Hamiltonian isotopy in the tubular neighbourhood isomorphic to $T^*L(\log \partial L)$, and so according to Proposition \ref{prop:logdef} $\alpha$ must have been an exact form. Thus $\tilde{\phi}_t$ was a Hamiltonian isotopy, and since $\beta^*$ is an isomorphism of elliptic and log cohomology, so was $\phi_t$. Thus $L,L'\subset (M,D)$ are related by Hamiltonian isotopy. 
\end{proof}

\begin{proof}[Proof of Theorem \ref{thm:defbdy}]
In Section \ref{sec:logdef} we have already shown that small deformations of $\tilde{L}=\beta^*(L)\subset \tilde{M}$ up to local Hamiltonian isotopy correspond to $H^1(\tilde{L},\tilde{L}\cap\tilde{D})$. By the definition above, small deformations of $\tilde{L}$ are clearly in one-to-one correspondence with small deformations of $L$. 

Since any elliptic Hamiltonian flow on $(M,D)$ will lift to a log Hamiltonian flow on $(\tilde{M},\tilde{D})$ via the lift of elliptic to log vector fields under real oriented blow-up, it is clear that if two Lagrangian branes with boundary in $(M,D)$ are Hamiltonian isotopic, their preimages in $(\tilde{M},\tilde{D})$ are, too. And thus, if one is a local deformation of the other, it corresponds to an exact log one-form, as long as the Hamiltonian isotopy is local. 

Conversely, we need to show that if a deformation is given as the graph of an exact log form on $\tilde{L}$ in $(\tilde{M},\tilde{D})$, its image in $(M,D)$ is related to the original brane by a smooth elliptic Hamiltonian isotopy. Although this is the simpler direction in the context of ordinary Lagrangians in symplectic manifolds, it turns out to be less intuitive here: Obviously, the graph of an exact log one-form $\dx f$ on $L$ as a Lagrangian in $T^*L(\log \partial L)$ is log Hamiltonian isotopic to the zero section via the flow of the log Hamiltonian vector field associated to $f$, but this Hamiltonian vector field does not descend as a smooth elliptic vector field to $(M,D)$. 

Instead, we begin by constructing a smooth symplectic isotopy between the original brane and the image of the graph of a sufficiently small closed log one-form: 
Let $\alpha' \in \Omega^1_{\operatorname{cl}}(L,\log \partial L)$.
 Note that there is always an extension $\alpha\in \Omega^1(U,\log \abs{D\cap U})$ to some wedge neighbourhood of $L$ such that $\iota^*_L(\alpha)=\alpha'$ and such that $\beta^*\alpha$ defines a map 
\[ \beta^*\alpha: U \rightarrow T^*L(\log \partial L), (\beta^*\alpha)(\xi_p) \in T^*_pL(\log \partial L).  \]
Namely, if $\alpha'= f_r(r,x,q_i) \frac{\dx r}{r} + f_x(r,x,q_i) \dx x +  f_{q_i}(r,x,q_i) \dx q_i$ in local coordinates $(r,x,q_i)$ on $L$, we pick the extension 
\[ \alpha= f_r(r\cos \theta,x,q_i) \frac{\dx r}{r} + f_x(r\cos \theta,x,q_i) \dx x +  f_{q_i}(r\cos \theta,x,q_i) \dx q_i, \] 
which is a smooth elliptic one-form on $(U,U\cap D)$, although it is of course not closed. This extension exists across all of $L$. 
Now consider the following isotopy of diffeomorphisms on a neighbourhood of the zero-section of $T^*L(\log \partial L)$, which descends to $(M,D)$ under the blow-down map: 
\[ \psi_t: \xi_p \mapsto \xi_p + t\alpha(\xi_p).  \]
One can check that $\psi_t$ are indeed diffeomorphisms, as long as we pick $(U,U\cap D)$ and $\alpha'$ to be sufficiently small. We have: 
$\psi_1(L)=\operatorname{Graph}(\alpha')$.
Of course, the $\psi_t$ do not preserve the elliptic symplectic form on $(U,U\cap D)$. Instead (for $\omega=\omega_0$ the standard local elliptic symplectic form on a wedge neighbourhood): $\psi_t^*\omega= \omega + t \dx \alpha$\\
Since $\dx \alpha\vert_L=0$, there is a choice of smooth $\bar{\alpha}\in \Omega^1(U,\log \abs{U\cap D})$ such that $\bar{\alpha}\vert_L=0$ and $\dx \bar{\alpha}=\dx \alpha$. 
Now we apply the relative Moser theorem using the flow $\phi_s$ of the elliptic vector field $\omega^{-1}(t\bar{\alpha})$ (for each $t$), which preserves $L$, and satisfies $\phi_t^*(\psi_t^* \omega)=\omega$. Thus, we have defined an elliptic symplectic isotopy between $L$ and the graph of $\alpha' \in \Omega^1_{\operatorname{cl}}(L,\log \partial L)$. 

If $\alpha'=\dx f$ is an exact log one-form on $L$, the existence of a symplectic isotopy between the two resulting branes with boundary in $(M,D)$ implies the existence of a Hamiltonian isotopy: See Lemma \ref{lem:symham}. 
\end{proof} 

\begin{rmk} 
Lagrangian branes with boundary are coisotropic submanifolds with respect to the Poisson structure $\omega^{-1}$. But in contrast to the Lagrangians which intersect the degeneracy locus transversely and whose deformations we discussed in Section \ref{sec:logdef}, we do not have a standard local form for the Poisson structure on a full tubular neighbourhood of the brane, only on a wedge neighbourhood. However, both the explicit computation of the $L_{\infty}$-structure, as well as the result on deformations of coisotropic submanifolds with respect to fibrewise entire Poisson structures require the Poisson structure to be known on a full tubular neighbourhood. Thus these results are at present not applicable to general Lagrangian branes with boundary. 

We can find examples where $\omega^{-1}$ is fibrewise entire on a full neighbourhood of a brane with boundary $(L,\partial L)$, and where the Maurer-Cartan elements of the $L_{\infty}$-structure on $NL$ do indeed again reduce to $\Omega^1_{\operatorname{cl}}(L,\log \partial L)$. 
\end{rmk} 

\section{Ehresmann connections for log symplectic Lefschetz fibrations} \label{sec:logehr}
Just like for ordinary symplectic structures, it makes sense to consider Lefschetz fibrations which admit a log symplectic structure. They will be Lefschetz fibrations over a surface with a marked hypersurface, and the logarithmic structure is such that the singular locus fibres over that hypersurface in the base. These fibrations have been defined and studied in detail in \cite{Cavalcanti2016} (using slightly different terminology than in this text): 
\begin{defn} 
A \emph{b-manifold} is a pair $(M,Z)$ of a manifold $M$ and a hypersurface $Z\subset M$, equipped with the logarithmic tangent bundle $TM(-\log Z)$. A b-manifold $(M,Z)$ is called \emph{b-oriented} if $TM(-\log Z)$ is oriented. 
A \emph{b-map} between two b-manifolds $(M,Z_M), (N, Z_N)$ is a map $f: M\rightarrow M$ such that $f^{-1}(Z_N)=Z_M$ and $f$ is transverse to $Z_N$. Write $f:(M,Z_M)\rightarrow (N,Z_N)$ . \\
A \emph{b-Lefschetz fibration} or \emph{logarithmic Lefschetz fibration} is a b-map $f: (X^{2n},Z_X)\rightarrow (\Sigma^2,Z_{\Sigma})$ between compact connected b-oriented b-manifolds such that for each critical point $x$ in the set $\Delta$ of all critical points there exist complex coordinate charts compatible with the orientations induced by the b-orientations, centred at $x$ and $f(x)$ in which $f$ takes the form 
\[ f: \Cnum^n \rightarrow \Cnum, (z_1,\dots,z_n)\mapsto z_1^2+\dots + z_n^2 \] 
\end{defn} 
\begin{rmk} \begin{enumerate}[label=(\roman*)]
\item Since a b-map $f$ is transverse to $Z_N$, i.e. $\operatorname{Im}(f)\pitchfork Z_N$, $f:(M,Z_M)\rightarrow (N,Z_N) $ induces a morphism 
\[ f_*: TM(-\log Z_M) \rightarrow TN(-\log Z_N) \] 
which maps $ f_*: R_M \twoheadrightarrow R_N, $
where $R_M=\ker (a_M\vert_{Z_M})\subset TM(-\log Z_M)\vert_{Z_M}, R_N=\ker a_N\vert_{Z_N}\subset TN(-\log Z_N)\vert_{Z_N}$. The respective canonical sections will be mapped to each other at every point. The reason that this is well-defined: Any defining function for $Z_N$ pulls back to a defining function for $Z_M$, because $f^{-1}(Z_N)=Z_M$.  
\item From the local model around a Lefschetz singularity $x \in M$ we can see that $\dx f\vert_x=0$, so since $f$ is transverse on $Z_M$, the set of critical points $\Delta$ and $Z_M$ are disjoint. 
\end{enumerate} 
\end{rmk} 

Let $f:(X,Z_X)\rightarrow (\Sigma,Z_{\Sigma})$ be a logarithmic Lefschetz fibration. There is a commutative diagram of vector bundles over $X$ ($V$ the vertical distribution of $f$): 

\begin{minipage}{\linewidth} \centering
\begin{tikzpicture}[scale=1.1]
\draw[->] (-4.8,0)--(-4.2,0);
\draw[->] (-3.8,0)--(-3.1,0) ;
\draw[->] (-1,0)--(-0.3,0) ;
\draw[->] (2.3,0)--(2.8,0) ;
\draw (-5,0) node{0} (-4,0) node{$V$} (-2,0) node{$TX(-\log Z_X)$} (1,0) node{$f^*(T\Sigma(-\log \Zsig))$} (3,0) node{0} ;
\draw (-4.1,-0.2)--(-4.1,-0.7) (-4,-0.2)--(-4,-0.7) ; 
\draw[->] (-2,-0.2)--(-2,-0.7); \draw[->] (1,-0.2)--(1,-0.7) ; 
\draw (-5,-1) node{0} (-4,-1) node{$V$} (-2,-1) node{$TX$} (1,-1) node{$f^*(T\Sigma)$} (3,-1) node{0} ;
\draw[->] (-4.8,-1)--(-4.2,-1);
\draw[->] (-3.8,-1)--(-2.5,-1) ;
\draw[->] (-1.5,-1)--(0.2,-1) ;
\draw[->] (1.8,-1)--(2.8,-1) ;
\draw (-0.5,0.2)node{$f_*$} ;
\draw (-0.5,-0.8)node{$f_*$} ;
\draw (-1.7,-0.5)node{$a_X$} (0.7,-0.5)node{$a_{\Sigma}$};


\end{tikzpicture} 
\end{minipage} 
See Proposition 2.14 in \cite{Cavalcanti2016} for a proof that $\ker( f_*: TX(-\log Z_X) \rightarrow T\Sigma(-\log \Zsig))$ and $\ker(f_*: TX\rightarrow T\Sigma)$ can indeed be identified via the anchor $a_X: TX(-\log Z_X)\rightarrow TX$. 

When restricted to $Z_X$, we obtain a diagram with exact rows and columns: 

\begin{minipage}{\linewidth} \centering
\begin{tikzpicture}[scale=1.3]
\draw (-2,2)node{0} (1,2)node{0} ;
\draw[->] (-2,1.8)--(-2,1.2); \draw[->] (1,1.8)--(1,1.2) ; 

\draw (-2,1)node{$R_X$} (1,1)node{$R_{\Sigma}$} ;
\draw[->] (-2,0.8)--(-2,0.2); \draw[->] (1,0.8)--(1,0.2) ; 
\draw[->] (-1.5,1)--(0.5,1); 
\draw (-0.5,1.2)node{$f_*$} ; 

\draw[->] (-4.8,0)--(-4.2,0);
\draw[->] (-3.8,0)--(-3.1,0) ;
\draw[->] (-1,0)--(-0.2,0) ;
\draw[->] (2.2,0)--(2.8,0) ;
\draw (-5,0) node{0} (-4,0) node{$V$} (-2,0) node{$TX(-\log Z_X)\vert_{Z_X}$} (1,0) node{$f^*(T\Sigma(-\log \Zsig)\vert_{\Zsig})$} (3,0) node{0} ;
\draw (-4.1,-0.2)--(-4.1,-0.7) (-4,-0.2)--(-4,-0.7) ; 
\draw[->] (-2,-0.2)--(-2,-0.7); \draw[->] (1,-0.2)--(1,-0.7) ; 
\draw (-5,-1) node{0} (-4,-1) node{$V$} (-2,-1) node{$TZ_X$} (1,-1) node{$f^*(T\Zsig)$} (3,-1) node{0} ;
\draw[->] (-4.8,-1)--(-4.2,-1);
\draw[->] (-3.8,-1)--(-2.5,-1) ;
\draw[->] (-1.5,-1)--(0.2,-1) ;
\draw[->] (1.8,-1)--(2.8,-1) ;
\draw (-0.5,0.2)node{$f_*$} ;
\draw (-0.5,-0.8)node{$f_*$} ; 

\draw (-2,-2)node{0} (1,-2)node{0} ;
\draw[->] (-2,-1.2)--(-2,-1.8); \draw[->] (1,-1.2)--(1,-1.8) ; 
\end{tikzpicture} 
\end{minipage}
\begin{prop}\label{prop:ext}
Assume that $\tilde{H}: T\Sigma(-\log \Zsig) \rightarrow TX(-\log Z_X)$ is a section of $f_*:TX(-\log Z_X)\rightarrow T\Sigma(-\log \Zsig)$. If this is such that $\tilde{H}\vert_{Z_X}(R_{\Sigma})=R_X$, $\tilde{H}$ induces an Ehresmann connection $H: T\Sigma \rightarrow TX$ (defined on all of $\Sigma$!). 
\end{prop} 
\begin{proof} 
First note that $TX\vert_ {X\setminus Z_X} \cong TX(-\log Z_X)\vert_{X\setminus Z_X}$ via the anchor map $a: TX(-\log Z_X) \rightarrow TX$ induced by the inclusion of log vector fields, and similarly $T\Sigma\vert_{\Sigma\setminus \Zsig} \cong T\Sigma(-\log \Zsig)\vert_{\Sigma\setminus \Zsig}$. Thus $\tilde{H}$ induces an Ehresmann connection $H:T\Sigma\vert_{\Sigma\setminus\Zsig} \rightarrow TX\vert_{X\setminus Z_X}$. So it suffices to show that this extends in a well-defined manner and smoothly to $H: T\Sigma \rightarrow TX$. 

Any splitting $s:TZ_{\Sigma}\rightarrow T\Sigma(-\log Z_{\Sigma})\vert_{Z_{\Sigma}}$ of the short exact sequence 
\[0 \rightarrow R_{\Sigma} \rightarrow T\Sigma(-\log \Zsig)\vert_{\Zsig} \rightarrow TZ_{\Sigma} \rightarrow 0 \] 
induces the same map $H: T\Zsig \rightarrow TZ_X$ that is compatible with the splitting outside $\Zsig$: The difference between the two splittings is in $R_{\Sigma}$, so by the assumption $\tilde{H}(R_{\Sigma})=R_X$, 
\[ H=a \circ \tilde{H}\circ s: TZ_{\Sigma} \rightarrow TZ_X \] 
does not depend on the choice of $s$. 

Consider a tubular neighbourhood of $\Zsig$, with $x'$ a local defining function for $\Zsig=\{x'=0\}$. Since $f$ defines a logarithmic Lefschetz fibration, $x=x'\circ f$ defines a local defining function for $Z_X$ on a tubular neighbourhood of $Z_X$. The vector field $x \party{}{x}$ is defined everywhere on the tubular neighbourhood and, as a section of $TX(-\log Z_X)$ its restriction to $Z_X$ generates $R_X$ (and similarly for $x'\party{}{x'}$ on $\Sigma$). \\ 
According to the assumption $\tilde{H}(R_{\Sigma})=R_X$, we have $\tilde{H}(x'\party{}{x'})=x\party{}{x}+x v$, where $v\in \Gamma(V)$ on the tubular neighbourhood. Locally, $\party{}{x'}$ is a normal vector field to $\Zsig$, which extends to the tubular neighbourhood, and the obvious induced Ehresmann connection \emph{outside} $\Zsig$ 
\[H: \party{}{x'} \mapsto \party{}{x} + v \] 
extends to $\Zsig$ itself. In the case where $N\Zsig$ is trivial (which automatically means $NZ_X$ is trivial, too), $x',x$ are coordinates on the entire tubular neighbourhoods of $\Zsig$ or $Z_X$ respectively, and so $\party{}{x'}, \party{}{x}$ are well defined as normal vector fields everywhere. In this case it is obvious that $H$ is well-defined. \\ 
If $N\Zsig$ is not orientable: With a chosen tubular neighbourhood embedding, the normal coordinates (= fibre coordinates for the normal bundle) on different patches around $\Zsig$ are related by multiplication with a non-zero function on $\Zsig$: 
\[ x' \mapsto gx', g\neq 0 \Rightarrow \party{}{x'} \mapsto \frac{1}{g}\party{}{x'} \] 
Since $gx'\party{}{gx'} =x'\party{}{x'}$, we must have $x\party{}{x} + xv = (f^*g)x \party{}{(f^*g)x} + (f^*g)x \bar{v}$, i.e. $\bar{v}=\frac{1}{g} v,$
on the new coordinate neighbourhood -- this makes $H$ as above consistent. 
\end{proof} 

Now assume that  $X$ is equipped with a log symplectic form $\omega$ s.t. the pullback of $\omega$ to $V=\ker f_*$ is non-degenerate, in particular the fibres of $f$  in $X\setminus Z_X$ are \emph{symplectic}. 
We call such b-Lefschetz fibrations \emph{log symplectic}. 
 Consider the unique splitting $\tilde{H}:  T\Sigma(-\log \Zsig) \rightarrow TX(-\log Z_X)$ s.t. the image of $\tilde{H}$ is the symplectic orthogonal of the vertical distribution $V$. 

\begin{prop}\label{prop:symlea} The following two statements are equivalent: 
\begin{enumerate}[label=(\roman*)]
\item On $Z_X$: 
$\tilde{H}(R_{\Sigma})=R_X$ 
\item The fibres of $f\vert_{Z_X}$ are are made up of leaves of the symplectic foliation of $\omega$ in $Z_X$. 
\end{enumerate}
\end{prop} 
\begin{proof}
(i)$\Rightarrow$ (ii): Since the image of $\tilde{H}$ is the symplectic orthogonal of $V$, we obtain $i_{R_X} \omega\vert_V=0$, i.e. $V\subset \ker(\res \omega)$. But $\ker(\res\omega)$ precisely defines the symplectic foliation of $\omega$ in $Z_X$, and since both it and $V$ have dimension $2n-2$, the fibres of $f\vert_{Z_X}$ must be made up of symplectic leaves of $\omega$. \\
(i)$\Leftarrow$ (ii):  Now by assumption $V\subset \ker(\res\omega)\Rightarrow R_X\subset V^{\omega-\text{orth}}$. Since the image of $\tilde{H}$ is the log symplectic orthogonal of $V$, there has to be a $Y \in T\Sigma(-\log \Zsig)\vert_{\Zsig}$ s.t. $\tilde{H}(Y)$ spans $R_X$. $\Rightarrow Y\in R_{\Sigma}$, and $\tilde{H}(R_{\Sigma})=R_X$. 
\end{proof} 
\begin{rmk} 
All logarithmic Lefschetz fibrations we want to consider satisfy the conditions of the Proposition: 
In \cite{Cavalcanti2016}, Theorem 3.4 and 3.7, a log symplectic structure $\omega_X$ for $f:(X^4,Z_X^2)\rightarrow (Y^2,Z_Y^1)$ (with orientable, compact, homologically essential fibres $F$ and compact base $Y^2$) is constructed from a log symplectic structure $\omega_Y$ for $(Y^2,Z_Y^1)$. 
This uses a closed, and non-degenerate fibrewise smooth form, as well as the pullback of $\omega_Y$. Since the logarithmic (singular) term is pulled back from the base, $\ker(\res\omega_X)$ will contain the tangent spaces to the fibres of the fibration, i.e. the tangent spaces to the symplectic leaves. 
\end{rmk} 

\begin{cor} 
Given a fibration as in Proposition \ref{prop:symlea}, a path $\gamma:[0,1]\rightarrow Y$ with $\gamma(1)\in Y_Z$ s.t. $\gamma$ intersects $Z$ transversely, and a Lagrangian submanifold $l\subset F$ of a (regular) fibre $F$ of $f$, the Lagrangian $L$ swept out by parallel transport of $l$ along $\gamma$ will intersect $Z_X$ transversely. 
\end{cor} 

\section{Stable generalized complex Lefschetz fibrations under blow-up} \label{sec:gclef}
Let $(M,D,\omega)$ be a manifold with elliptic divisor $D$. Let $(\tilde{M},\tilde{D},\tilde{\omega})$ be the real oriented blow-up of $D$. Let $\beta: \tilde{M}\rightarrow M$ be the blow-down map s.t. $\beta\vert_{\tilde{D}}:\tilde{D}\rightarrow D$ is a $U(1)$-principal bundle. 

Analogously to logarithmic Lefschetz fibrations, there is a notion of Lefschetz fibration for manifolds equipped with an elliptic divisor. \cite{Cavalcanti2017} define and study these so-called \emph{boundary Lefschetz fibrations} in detail as a specific case of \emph{Lie algebroid Lefschetz fibrations}. Let $(\Sigma, Z)$ be a surface with a separating hypersurface (i.e. a line) $Z$. 
\begin{defn} 
A \emph{strong map of pairs} $f:(M,D)\rightarrow (\Sigma,Z)$ is a map with $f^{-1}(Z)=D$. 

Let $f:(M,D)\rightarrow (\Sigma,Z)$ be a map of pairs such that $\operatorname{Im}(f_*)\subset TZ$. The \emph{normal Hessian} of $f$ along $D$ is the map 
\[ H^{\nu}(f): \operatorname{Sym}^2(ND)\rightarrow f^*(NZ) \] 
which associates to $f$ its normal Hessian at each point: Since $\operatorname{Im}(f)\subset TZ$, the map 
\[ \nu(\dx f): ND\rightarrow NZ\] 
is the zero map. We can consider a local defining function $z$ for $Z$ and set $h:=f^*z$, which satisfies $\dx h\vert_D=0$. $H^{\nu}(f)$ is defined as the Hessian of this function at each point; it is easy to check that this is independent of the chosen defining function $z$. \\
A strong map of pairs $f: (M,D)\rightarrow (\Sigma,Z)$ is called a \emph{boundary map} if its normal Hessian $H^{\nu}(f)$ is definite along $D$. 
A boundary map $f$ is called \emph{fibrating} if $f\vert_D:D\rightarrow Z$ is a submersion. \\
A \emph{boundary Lefschetz fibration} is a fibrating boundary map $f$ such that $f\vert_{X\setminus D}: X\setminus D\rightarrow \Sigma\setminus Z$ is a Lefschetz fibration (see Definition 5.21 in \cite{Cavalcanti2017}). 
\end{defn}
\begin{rmk} \begin{enumerate}[label=(\roman*)]
\item Any fibrating boundary map $f:(M,D)\rightarrow (\Sigma,Z)$ is a submersion in a punctured neighbourhood around $D$. (Of course, $f$ is \emph{not} submersive on a full open neighbourhood of $D$.)
\item By passing to a cover of $\Sigma$, we can always assume that the generic fibres of a boundary Lefschetz fibration are connected. If the generic fibres near $D$ are connected, the fibres of $f\vert_D: D\rightarrow Z$ are also. 
\end{enumerate} 
\end{rmk}  

First note the following (see \cite{Cavalcanti2017}): If $f:(M,D)\rightarrow (\Sigma,Z)$ is a boundary map to the surface $\Sigma$ such that $Z$ is a separating submanifold, $NZ$ is in particular orientable and there is a global defining function $z$ for $Z$ s.t. $f(X)\subset \Sigma_+$, the locus where $z\geq 0$. Then $f$ defines a boundary map $f:(M,D)\rightarrow (\Sigma',Z')$, where $\Sigma'=\Sigma_+\cap f(X),Z'=Z\cap f(D)$. So if $Z$ is separating, we can always assume that it is in fact the boundary of $\Sigma$. 

Furthermore, whenever $(\Sigma,Z)$ is any manifold admitting a log symplectic structure which is also oriented, $Z$ is separating. 

\begin{prop} \label{prop:Lefschetzbu}
Assume that $f:(M,D)\rightarrow (\Sigma,Z)$ is a boundary Lefschetz fibration over a surface with boundary $(\Sigma, \partial\Sigma=Z)$. Then there exists a real branched double cover of $(\Sigma,Z)$ over $\tilde{Z}\cong Z$, $\rho: (\tilde{\Sigma},\tilde{Z}) \rightarrow (\Sigma,Z)$ (in which $\tilde{Z}$ is separating) s.t. $f$ factors through $(\tilde{\Sigma},Z)$: 
$ f= \rho \circ f'$, and 
\[f' \circ\beta: (\tilde{M},\tilde{D})\rightarrow (\tilde{\Sigma},\tilde{Z}) \] 
is a logarithmic Lefschetz fibration. 
\end{prop} 

\paragraph{\textbf{Branched cover over a surface with boundary}} Around every boundary component of $\Sigma$ (either $\Rnum$ or $S^1$), there is a collar neighbourhood s.t. the boundary component is given by the vanishing of a positive coordinate. Denote this coordinate by $x$. For simplicity, write $Z$ for a single boundary component. 

Now, a trivial branched cover of $\Sigma$ over the boundary $Z$ can be constructed as follows: Consider $Z\times [0,1) \times \Rnum$ (where the first two factors are a collar neighbourhood of $Z$ in $\Sigma$). There is a smooth surface defined by the graph of $y^2=x$ inside $Z\times [0,1) \times \Rnum \ni (z,x,y)$. Its closure has boundary $Z\times \{\pm1\}$, so we have doubled the previous boundary of the collar neighbourhood. 
To obtain the full branched double cover $\tilde{\Sigma}$, glue one copy each of $\Sigma\setminus(Z\times [0,1))$ to each \enquote{arm} of the new branched surface. If there are multiple boundary components, glue one to all the $+1$-boundaries, the other to the $-1$-boundaries. Up to diffeomorphism, this procedure defines a unique new surface $\tilde{\Sigma}$ without boundary, in which $Z$ is a separating hypersurface. $\rho: \tilde{\Sigma}\rightarrow \Sigma$, locally around $Z$ defined by 
\[ (z,y^2,y)\mapsto (z,y^2)\ (z\in Z) \] 
and away from the boundary components by the identity on each leaf, defines a branched covering. 

\begin{proof}[Proof of Proposition  \ref{prop:Lefschetzbu}.]
Let $\rho: (\tilde{\Sigma},\tilde{Z})\rightarrow (Z,\Sigma)$ be a branched double cover as just constructed. \cite{Cavalcanti2017} prove a standard local form for boundary Lefschetz fibrations: There are coordinates $(r,\theta, x_3,\dots,x_{2n})$ around $D$ in $(M,D)$ and $(x,z)$ around $Z$ in $(\Sigma, Z)$ ($x$ a defining function for $Z$) such that the boundary Lefschetz fibration $f: (M,D)\rightarrow (\Sigma,Z)$ near $D$ takes the form 
\[ f(r,\theta,x_3,\dots,x_{2n})= (r^2,x_{2n}) \] 
This factors through the $+1$-arm of $(\tilde{\Sigma},\tilde{Z})$ with coordinates $(y,z)$ around $\tilde{Z}$ ($y$ a defining function for $\tilde{Z}$) as 
\[ f: (r,\theta,x_3,\dots,x_{2n})\stackrel{f'}{\rightarrow} (r,x_{2n}) \stackrel{\rho}{\rightarrow} (r^2,x_{2n}) \] 
If we compose this $f'$ with the blow-down map $\beta$, we obtain a logarithmic Lefschetz fibration 
\[ \tilde{f}=f' \circ \beta: (\tilde{M},\tilde{D})\rightarrow (\tilde{\Sigma},\tilde{Z}), \] 
since $\tilde{f}^{-1}(\tilde{Z})=\tilde{D}$. $\beta$ is a submersion, and $f'_*(\party{}{r})=\party{}{x}$, so $\tilde{f}$ is transverse to $\tilde{Z}$. 
\end{proof} 

\section{Lefschetz thimbles in stable generalized complex Lefschetz fibrations} \label{sec:thimb}
From the previous section we know that a boundary Lefschetz fibration $f:(M,D)\rightarrow (\Sigma,Z)$ blows up to a log Lefschetz fibration $\tilde{f}:(\tilde{M},\tilde{D}) \rightarrow (\Sigma,Z)$. Assume that $(M,D)$ is equipped with a stable GC structure given by the elliptic symplectic form $\omega$ in such a way that $\omega$ is non-degenerate on $\ker(f_*: TM(-\log \abs{D})\rightarrow T\Sigma(-\log Z))$, in particular the fibres of $f$ in $M\setminus D$ are symplectic. (Similarly to log Lefschetz fibrations, the fibres of a boundary Lefschetz fibration are either entirely in $D$ or entirely in $M\setminus D$.) We call this a \emph{stable generalized complex Lefschetz fibration}. 

\begin{prop} \begin{enumerate}[label=(\roman*)]
\item A stable GC Lefschetz fibration $f:(M,D)\rightarrow (\Sigma,Z)$ induces the structure of a log symplectic Lefschetz fibration on $\tilde{f}: (\tilde{M},\tilde{D})\rightarrow (\Sigma,Z)$. 
\item If $\tilde{H}: f^*(T\Sigma(-\log Z)) \rightarrow TM(-\log\abs{D})$ is the elliptic Ehresmann connection associated to $\omega$, we have: 
\[ \tilde{H}(R_{\Sigma})=R_M \Leftrightarrow \text{The fibres of }f\vert_D \text{ are given by }\ker(\res_r\omega), \]
where $R_M\subset \ker a, a: TM(-\log\abs{D})\vert_D \rightarrow TD$ is the subbundle spanned by the Euler vector field on $ND$. 
\end{enumerate} 
\end{prop} 
\begin{proof} 
\begin{enumerate}[label=(\roman*)] 
\item From the real oriented blow-up of $D$ inside $M$, we obtain the following commutative diagram with exact rows: 

\begin{minipage}{\linewidth} \centering
\begin{tikzpicture}[scale=0.8]
\draw (-5,0)node{0} (-3,0)node{$\tilde{V}$} (0,0)node{$T\tilde{M}(-\log\tilde{D})$} (4,0)node{$\tilde{f}^*(T\Sigma(-\log Z))$} (7,0)node{0};
\draw[->] (-4.8,0)--(-3.5,0); \draw[->] (-2.5,0)--(-1.5,0) ;  
\draw[->] (1.5,0)--(2.3,0) ;  \draw[->] (5.7,0)--(6.8,0) ; 
\draw[->] (-3,-0.5)--(-3,-1.5); \draw[->] (0,-0.5)--(0,-1.5) ; \draw[->] (4,-0.5)--(4,-1.5) ; 
\draw  (-5,-2)node{0} (-3,-2)node{$V$} (0,-2)node{$TM(-\log\abs{D})$} (4,-2)node{$f^*(T\Sigma(-\log Z))$} (7,-2)node{0};
\draw[->] (-4.8,-2)--(-3.5,-2); \draw[->] (-2.5,-2)--(-1.5,-2) ;  
\draw[->] (1.5,-2)--(2.3,-2) ;  \draw[->] (5.7,-2)--(6.8,-2) ; 
\draw (-0.5,-1)node{$\beta_*$} (3.5,-1)node{$\rho_*$} (4.5,-1)node{$\cong$} ;
\end{tikzpicture} 
\end{minipage}
Note that $\beta_*$ is fibrewise an isomorphism, and it induces a fibrewise isomorphism on $\tilde{V}$. Thus if $\omega$ is non-degenerate on $V$, $\tilde{\omega}=\beta^*\omega$ is non-degenerate on $\tilde{V}$.
\item \enquote{$\Rightarrow$}: By definition of $\tilde{H}$, $\operatorname{Im}(\tilde{H})$ is the symplectic orthogonal of $V$ in $TM(-\log \abs{D}))$, so $i_{R_M}\omega \vert_V=0$. Consider $V'=a(V)\subset TD$. We have $\res\omega(v)=\iota_D^*(i_{R_M}\omega)(v)=0\ \forall v\in V'$, so $V'\subset \ker(\res\omega)$. But both $\ker(\res\omega)$ and $V'$ have rank $2n-3$, so they are equal. \\ 
\enquote{$\Leftarrow$}: Now we assume that the fibres of $f\vert_D$ have the tangent distribution $V'=\ker(\res\omega)$. Since the elliptic residue of $\omega$ is zero, this implies $R_M\subset V^{\omega-\text{orth}}$. Thus there exists $Y\in T\Sigma(-\log Z)$ such that $\tilde{H}(Y)$ spans $R_M$, and since $f_*(R_M)=R_{\Sigma}, Y$ spans $R_{\Sigma}$, so $\tilde{H}(R_{\Sigma})=R_M$. 
\end{enumerate} 
\end{proof}

\paragraph{\textbf{Lefschetz Thimbles}} 
Let $f: (M,D,\omega)\rightarrow (\Sigma,Z)$ be a stable generalized complex Lefschetz fibration whose fibres in $D$ correspond to $\ker(\res_r \omega)$. Consider the associated log symplectic Lefschetz fibration $\tilde{f}: (\tilde{M},\tilde{D},\tilde{\omega})\rightarrow (\tilde{\Sigma},\tilde{Z})$. According to what we have just shown, this admits an Ehresmann connection $H$ induced by its log symplectic form. 

Like for ordinary symplectic Lefschetz fibrations, we can consider \emph{Lefschetz thimbles} with respect to $H$ over a path in the base surface which ends at the image of one of the Lefschetz singularities in the interior (i.e. a critical value of $\tilde{f}$). 
 If the path is chosen such that it hits $Z$ transversely, the associated Lefschetz thimble will be a logarithmic Lagrangian that intersects $\tilde{D}$ transversely in its (spherical) boundary. 
\begin{figure}[h] \centering 
\includegraphics[scale=0.35]{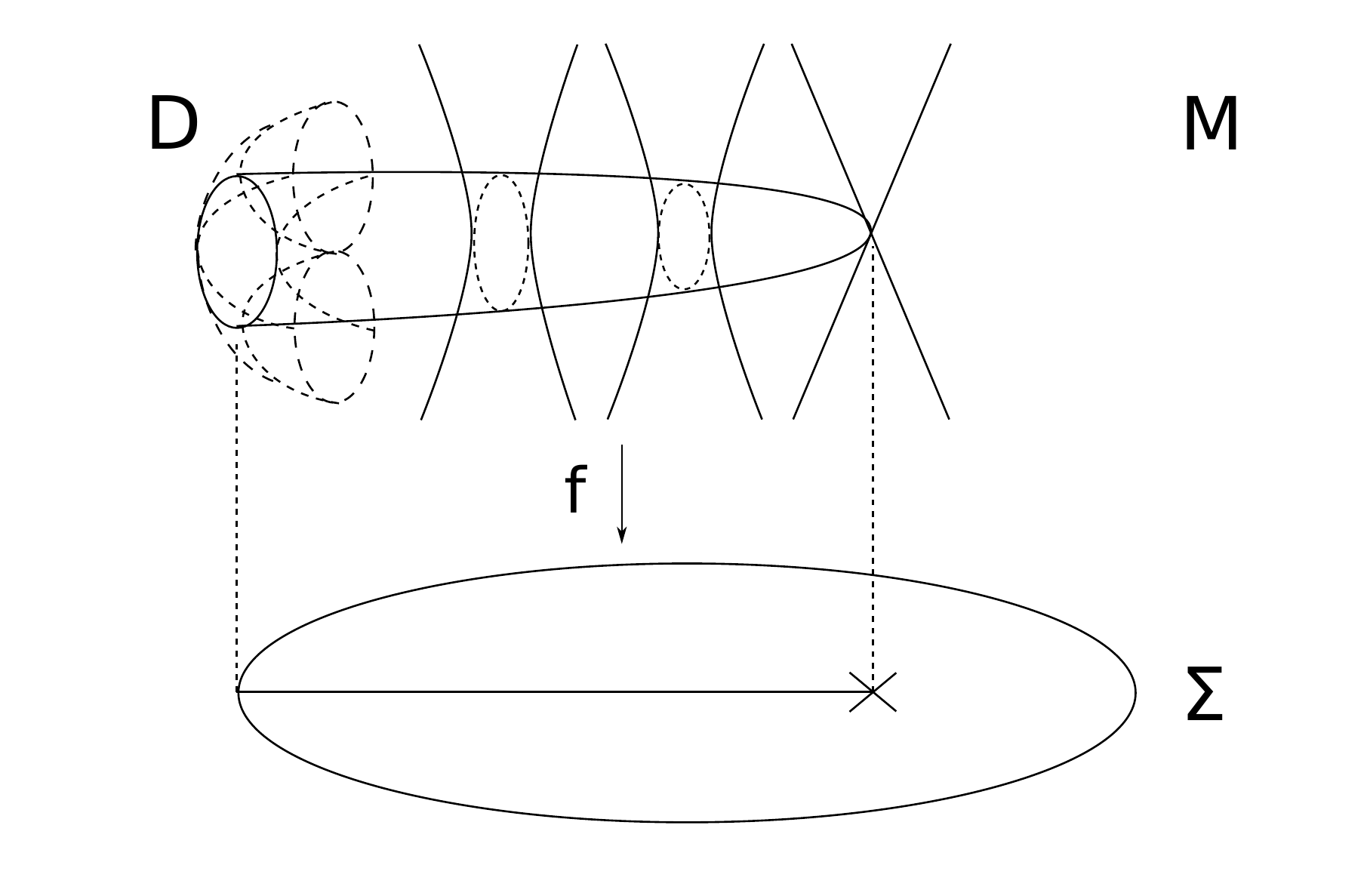}
\caption{A Lefschetz thimble which is a brane with boundary}
\end{figure}

Now, when looking at the image of such a thimble under the blow-down map $\beta$, there are two main cases of interest according to Theorem \ref{thm:branes}:  
\begin{enumerate} 
\item If $\party{}{\theta}$ is not tangent to and $\beta$ is injective on the boundary of the Lefschetz thimble, the image in $(M,D)$ will be a Lagrangian brane with boundary, a Lefschetz thimble whose boundary lies in the anticanonical divisor. 
\item If the $U(1)$-action of $\tilde{D}$ restricts to the boundary sphere of the thimble, the thimble blows down to a Lagrangian brane without boundary (not always smooth, but always Hamiltonian isotopic to a Lagrangian that does blow down smoothly). \\
In the case where the total space of the fibration is a $4$-manifold, \cite{Behrens2017} define the \emph{boundary vanishing cycle} associated to the singular locus of a boundary Lefschetz fibration. This case then precisely occurs when the boundary vanishing cycle is the same as the vanishing cycle of the Lefschetz singularity at the other end of the thimble. The result is a Lagrangian brane which is topologically an $S^2$. 
(By choosing the correct base path, we can always obtain a smooth Lagrangian $S^2$.)   
\end{enumerate} 

In the following, we examine some examples of Lagrangian branes produced by the parallel transport of Lagrangian spheres in the fibres of stable generalized complex Lefschetz fibrations: 
 
\subsection{Example: The Hopf surface}\label{sec:Hopf}
Consider the complex manifold $X=\left(\Cnum\setminus\{0\}\right)/(z\sim 2z)$. This is clearly diffeomorphic to $S^3\times S^1$, viewing $S^3$ as 
\[ S^3= \left\{(z_0,z_1)\in \Cnum^2\vert \abs{z_0}^2+\abs{z_1}^2=1\right\} \] 
\begin{figure}[h] \centering 
\includegraphics[scale=0.3]{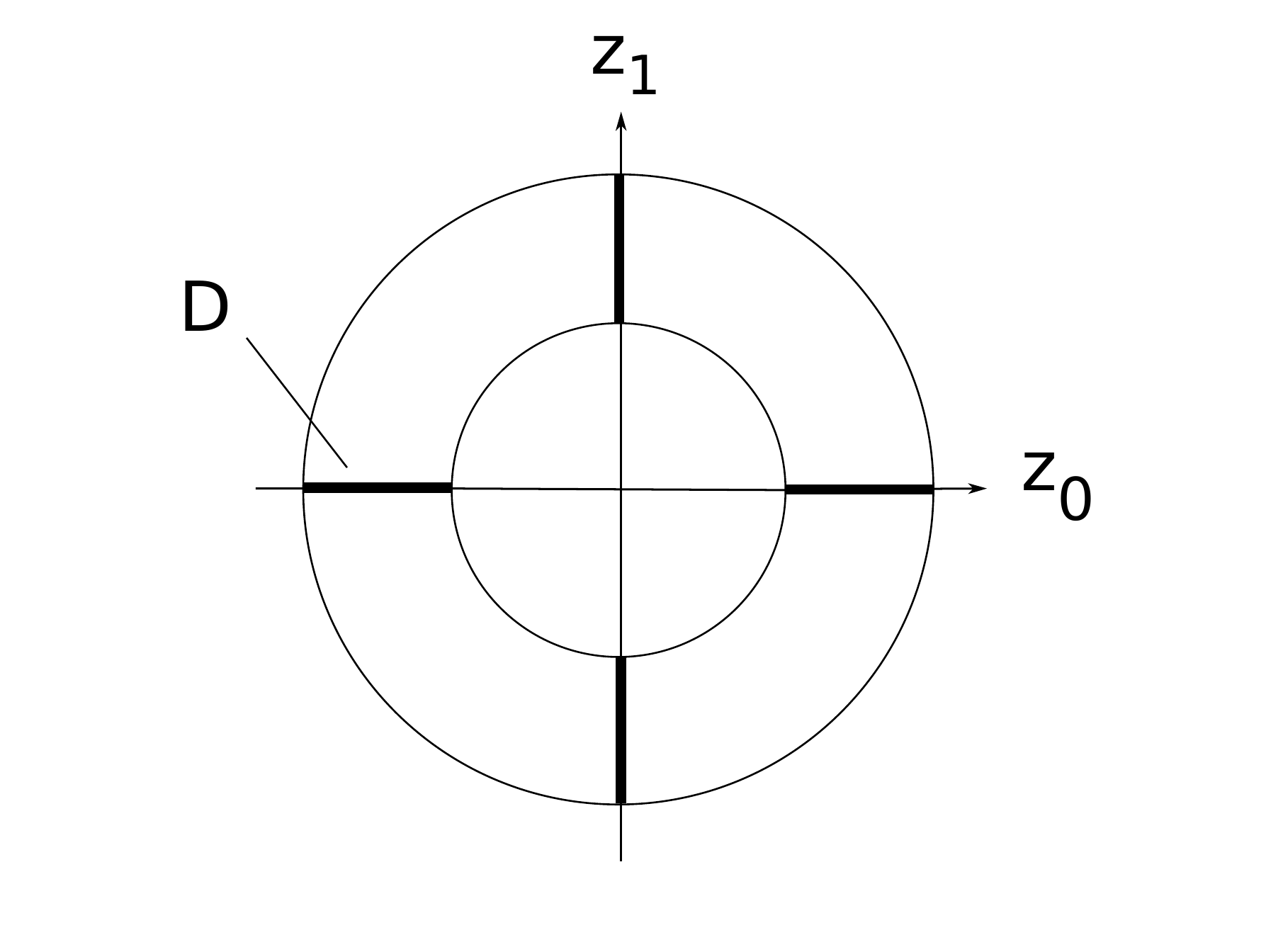}
\caption{Hopf surface in $\Cnum^2\setminus \{0\}$ with anticanonical divisor}
\end{figure} 
We can make $X$ into a boundary Lefschetz fibration without any singular fibres as follows \mbox{(see \cite{Cavalcanti2017}):} 
Compose the Hopf fibration $p: S^3 \rightarrow S^2, (z_0,z_1)\mapsto [z_0:z_1]$ with the standard height function $h:S^2 \rightarrow I=[0,1]$. 
Furthermore consider the $S^1$-coordinate given by 
\begin{equation} \eta=\frac{1}{2} \log \left(\abs{z_0}^2+\abs{z_1}^2\right). \end{equation}
Since $\abs{z_0}\sim 2\abs{z_0}, \abs{z_1}\sim 2\abs{z_1}$, we obtain $\eta \sim \eta + \log 2$. 
 Then 
\[ f(z_0,z_1) = (h([z_0:z_1]),\eta) \] 
defines a boundary Lefschetz fibration $F: X \rightarrow I\times S^1$.  

If $z_0=r_0 e^{i \theta_0}, z_1=r_1 e^{i\theta_1}$ and $t=\frac{r_0}{r_1}$, $(t,\eta,\theta_0,\theta_1)$ are coordinates for $M$ away from $z_1=0$. $(1/t,\eta,\theta_0,\theta_1)$ are coordinates away from $z_0=0$. 

We can pick coordinates $t',1/t'$ on $I$ s.t. the height function maps $(t,\theta_0-\theta_1)\in S^2$ to $t^2\in I$, and similarly on the other coordinate patch. In these coordinates, $f$ becomes $f(t,\eta,\theta_0,\theta_1)=(t^2,\eta)$, i.e. the fibres of $f$ are precisely the $(\theta_0,\theta_1)$-tori. 

\[  \frac{1}{2} \frac{\dx t'}{t'}\wedge \dx \eta = -\frac{1}{2} \frac{\dx (1/t')}{1/t'} \wedge \dx \eta \] 
is a well-defined logarithmic symplectic form on $I\times S^1$. We can pull it back to $M$ via $f$ to obtain 
\[ \frac{\dx t}{t}\wedge \dx \eta = - \frac{\dx (1/t)}{1/t} \wedge \dx \eta \]
This can be completed to the following elliptic symplectic form on $M$: 
\begin{equation} 
\omega= \frac{\dx t}{t}\wedge \dx \eta - \dx \theta_0 \wedge \dx \theta_1 = - \frac{\dx (1/t)}{1/t} \wedge \dx \eta - \dx \theta_0 \wedge \dx \theta_1 
\end{equation} 
This is actually the imaginary part of the holomorphic log symplectic form 
\begin{equation} 
\Omega= i\frac{\dx z_0 \wedge \dx z_1}{z_0 z_1} = \frac{\dx r_1}{r_1} \wedge \dx \theta_0 - \frac{\dx r_0}{r_0} \wedge \dx \theta_1 + i \left( \frac{\dx r_0}{r_0}\wedge \frac{\dx r_1}{r_1} - \dx \theta_0 \wedge \dx \theta_1 \right)
\end{equation} 
The anticanonical divisor of $\Omega$ is $D=\{z_0=0\}\cup\{z_1=0\}$. 
The symplectic orthogonal distribution to the fibres is spanned by 
\[ \left\{ t \party{}{t}=-1/t\party{}{(1/t)}, \party{}{\eta}\right\} \] 
and the corresponding Ehresmann connection $ \tilde{H}: T(I\times S^1)(-\log (\{0\}\times S^1\cup \{1\}\times S^1)) \rightarrow TM(-\log \abs{D}) $
lifts 
\[ t' \party{}{t'} \mapsto t \party{}{t}, \party{}{\eta} \mapsto \party{}{\eta}, \] 
so induces an Ehresmann connection  $H: T(I\times S^1) \rightarrow TM$. 

Now consider any path $\gamma$ in $I\times S^1$ from $(0,\eta_0)$ to $(1,\eta_0)$ with $\eta(s)=\eta_0$ constant. Under $H$, this lifts to a corresponding path with $t(s), \eta=\eta_0, \theta_0=\text{const.}, \theta_1=\text{const.}$
We can for example obtain the following Lagrangian branes with boundary from parallel transporting Lagrangian circles in the fibres along such a path in the base: 
\begin{ex}\label{ex:Hopf}
\begin{enumerate}[label=(\roman*)]
\item Circle with $\theta_0=\text{const.}, \theta_1 \in [0,2\pi)$: When parallel-transported to the component of $D$ with $z_1=0$, this circle closes up. Parallel-transporting all along the path $\gamma$ yields a Lagrangian brane with boundary that is diffeomorphic to $D^2$ and intersects $\{z_0=0\}$ in its circular boundary, $\{z_1=0\}$ in a point. 
Evidently we can exchange $z_0$ and $z_1$ to obtain a similar $D^2$-brane with boundary in $\{z_1=0\}$. 
\item Circle with $\theta_0=\theta_1=\theta\in [0,2\pi)$: When parallel-transported all along the path $\gamma$, we obtain a cylindrical Lagrangian brane with boundary, which intersects both $\{z_0=0\}$ and $\{z_1=0\}$ in a circle. 
\begin{figure}[h] \centering \label{fig:1win}
\includegraphics[scale=0.2]{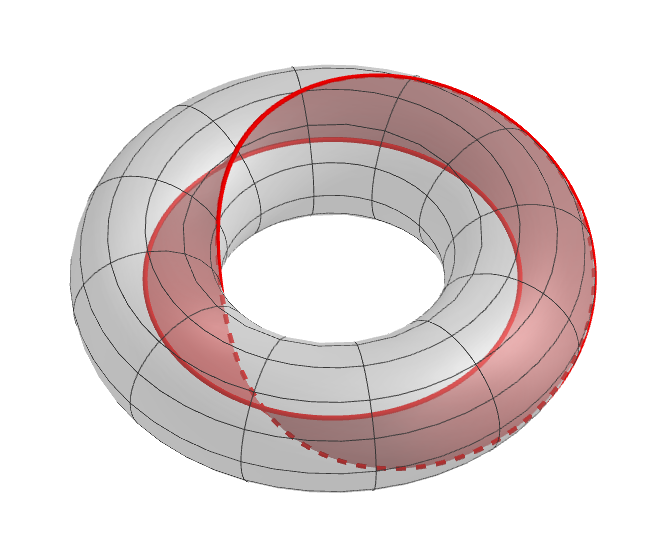}
\caption{Lagrangian (ii) in a neighbourhood of either component of the anticanonical divisor.}
\end{figure} 
\item Circle with $2\theta_0=\theta_1=\theta \in [0,2\pi)$: When parallel-transported all along the path $\gamma$, we obtain a smooth Lagrangian brane with boundary that is topologically a Möbius band. It intersects the $\{z_1=0\}$-locus in a circle which is its boundary, and $\{z_0=0\}$ also in a circle, the zero section of the Möbius band as a subset of the Möbius line bundle. 
Note that when lifted to the real oriented blow-up, this Lagrangian intersects both components of the singular locus in a circle, but on the blow-up of $\{z_0=0\}$ the blow-down map is not injective. 
\begin{figure}[h]\centering \label{fig:2win} 
\begin{tabu} to \linewidth {XX}
\includegraphics[scale=0.2]{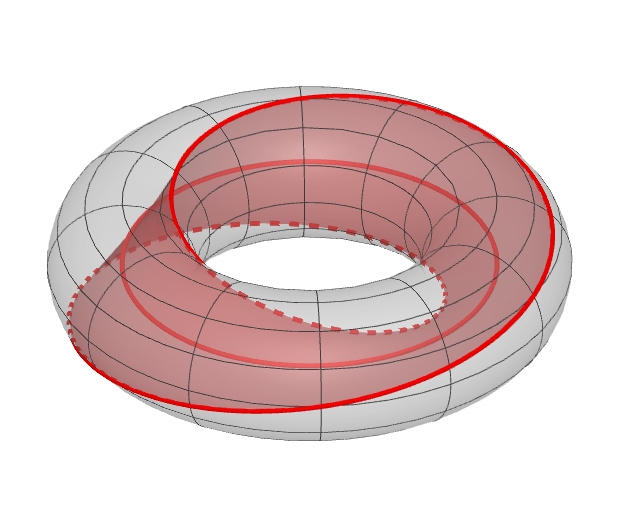} &
\includegraphics[scale=0.2]{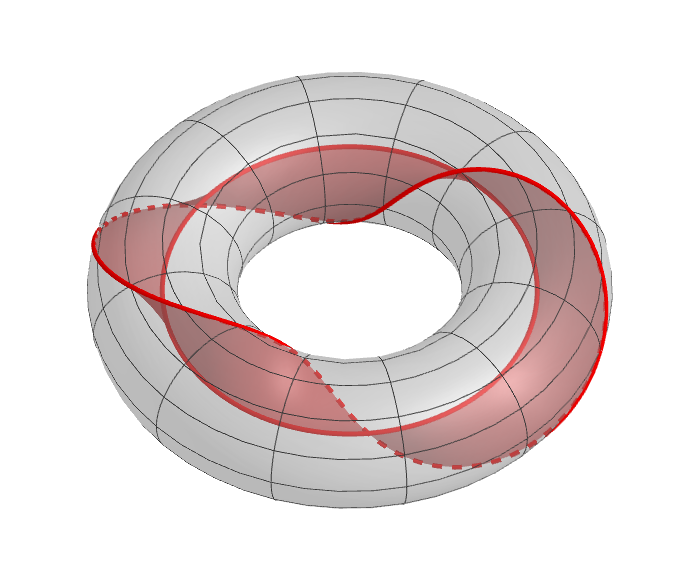} \\ 
\end{tabu} 
\caption{Lagrangian (iii) in an open neighbourhood of both components of the anticanonical divisor: On the left, we see the neighbourhood of $\{z_0=0\}$ and on the right the neighbourhood of $\{z_1=0\}$.}
\end{figure}
\item Circle with $3\theta_0=\theta_1=\theta \in [0,2\pi)$: When parallel-transported all along $\gamma$, this does not result in a smooth submanifold: Away from $\{z_0=0\}$, this is an open Lagrangian cylinder which will intersect $\{z_1=0\}$ in its circular boundary, but at $\{z_0=0\}$ there is a triple intersection of leaves of the cylinder. 
As in the previous case, the blow-down map is not injective on the boundary of the lift of this Lagrangian to the real oriented blow-up. 
\begin{figure}[h]\centering \label{fig:3win} 
\begin{tabu} to \linewidth {XX}
\includegraphics[scale=0.2]{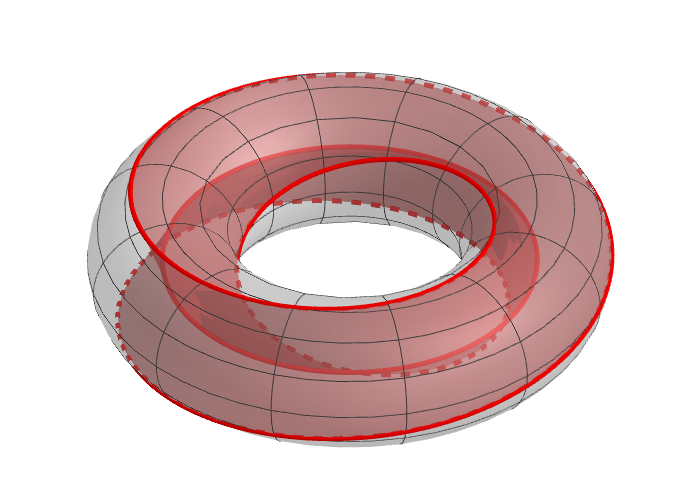} &
\includegraphics[scale=0.2]{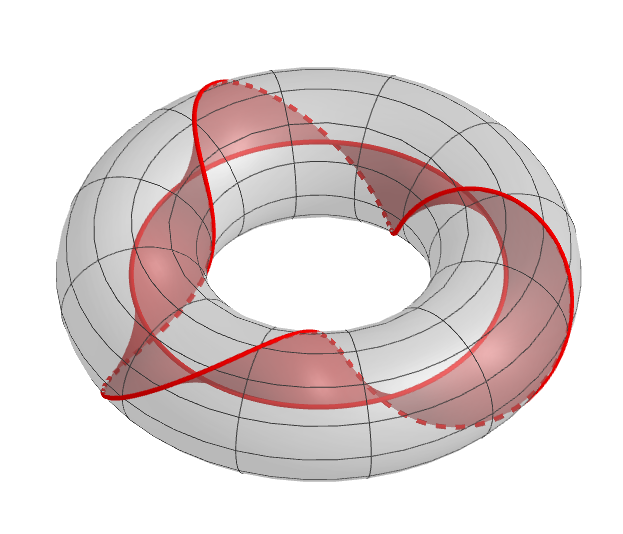} \\ 
\end{tabu} 
\caption{Lagrangian (iv) in an open neighbourhood of both components of the anticanonical divisor: On the left, we see the neighbourhood of $\{z_0=0\}$ and on the right the neighbourhood of $\{z_1=0\}$.}
\end{figure}
\end{enumerate} 
\end{ex} 

\subsection{Examples of genus-one boundary Lefschetz fibrations over the disk} 
\begin{ex} \textbf{Boundary Lefschetz fibration with one Lefschetz singularity.} 
Example 8.4 in \cite{Cavalcanti2017} describes the following scenario: If we consider the $4$-dimensional genus-1 Lefschetz fibration over $D^2$ with one singular fibre with vanishing cycle $b\in H^1(T^2)$, a generator, the monodromy around $\partial D^2$ is the Dehn twist with $b$. Thus this Lefschetz fibration can be completed to a boundary Lefschetz fibration with a stable generalized complex structure whose anticanonical divisor fibres over $\partial D^2$. 

From Proposition 6.2 and 6.5 in \cite{Cavalcanti2017} we obtain coordinates $(s,x,y,z)$ for a neighbourhood of the anticanonical divisor where $s$ is a radial coordinate for the distance from the anticanonical divisor, and $(x,y,z)$ angular coordinates such that 
\begin{align*} 
(x,y,z) &\sim (x,y+1,z) \\
(x,y,z)&\sim (x,y,z+1) \\ 
(x,y,z)&\sim (x+1,y,z-y). 
\end{align*} 
The projection to a tubular neighbourhood of $\partial D^2$ is $(s,x,y,z)\mapsto (s^2,x)$; $(y,z)$ are angular coordinates for the torus fibres. (These coordinates are for what is referred to as the \emph{standard 1-model}.) 
Note that the $z$-coordinate encodes the vanishing cycle $b$. 
On the other hand, $z$ is the angular coordinate in the fibre of the complex line bundle over $D$ that defines the standard 1-model. 

Thus as $r\rightarrow 0$, the $z$-circle shrinks to zero. So any Lefschetz thimble for the single Lefschetz singularity in this example over a path in the base from the singularity to the boundary will be topologically an $S^2$: The vanishing cycle sweeps out a disk when moving along a path away from the singularity, which closes up to a sphere as the vanishing cycle shrinks back to a point when approaching the anti-canonical divisor. 

As described in \cite{Cavalcanti2009} and \cite{Cavalcanti2017}, such spheres can be blown down in a way that is compatible with the stable generalized complex structure. After the blow-down, we obtain another boundary Lefschetz fibration for the Hopf surface. 
\end{ex} 

\begin{ex} \label{ex:3singu}
\begin{figure}[h] \centering 
\includegraphics[scale=0.5]{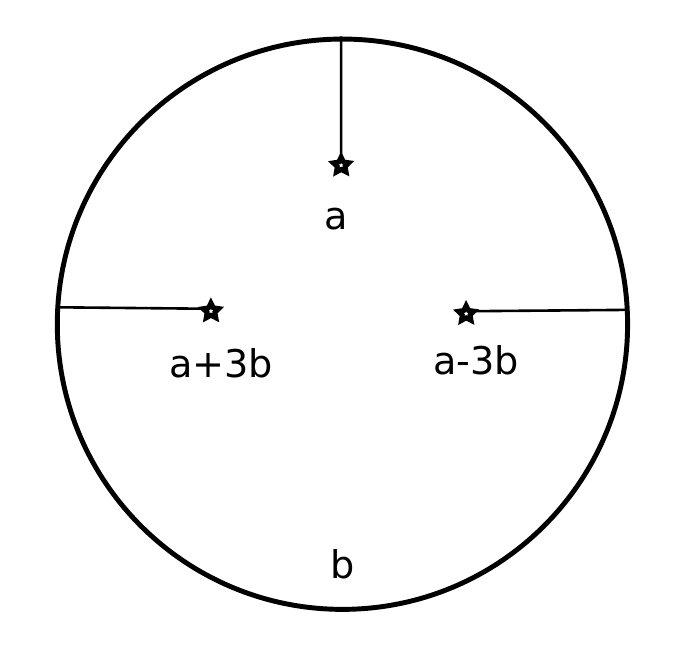}
\caption{Lefschetz singularities in Example \ref{ex:3singu}}
\end{figure} 
There are other genus-1 Lefschetz fibrations over the open disk with multiple Lefschetz singularities, but whose monodromy is still the power of a Dehn twist -- and whenever this is the case, they can be completed to a closed boundary Lefschetz fibration (see Proposition 6.5 in \cite{Cavalcanti2017}). 
For example (Example 8.5 in \cite{Cavalcanti2017}), if $a,b$ are generators of $H_1(T^2)$, there is a Lefschetz fibration with three singularities in the interior of the disk with associated vanishing cycles (in counter-clockwise order) 
\[ a-3b, a, a + 3b. \] 
The global monodromy around all three singularities is $9b$, so we can complete this to a boundary Lefschetz fibration by gluing in the standard 9-model $\operatorname{tot}(L_9)$ (see Proposition 6.5 in \cite{Cavalcanti2017}), such that the total space admits a stable generalized complex structure. Topologically, the resulting closed total space of this boundary Lefschetz fibration is $\Cnum P^2$ (see also Example 5.3 in \cite{Cavalcanti2009}).

In order to extend the Lefschetz thimbles associated to the three Lefschetz singularities into the anticanonical divisor, we follow the $C^{\infty}$-log surgery as described in Section 4 of \cite{Cavalcanti2009}: We consider the honest Lefschetz fibration of $\Cnum P^2 \# 9 \overline{\Cnum P^2}$ over $S^2$, with paths from all $3+9$ Lefschetz singularities to a regular reference fibre. We trivialise this Lefschetz fibration around the reference fibre and perform a $C^{\infty}$-log transform to obtain a generalized complex Lefschetz fibration over the disk. 
\begin{figure}[h] \centering 
\includegraphics[scale=0.5]{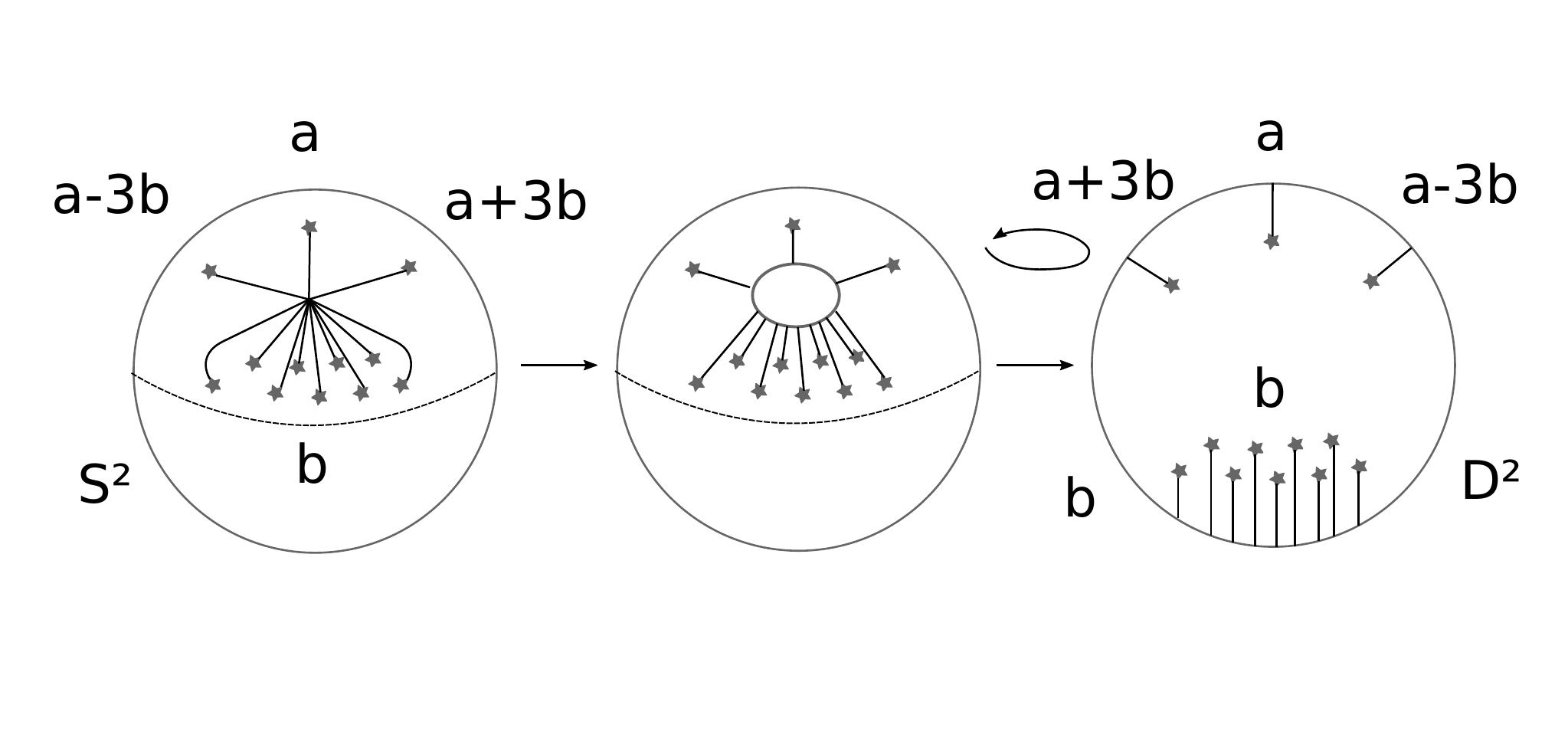}
\caption{Performing a $C^{\infty}$-log transform on the Lefschetz fibration of $\Cnum P^2 \# 9 \overline{\Cnum P^2}$ over $S^2$ to obtain a generalized complex Lefschetz fibration over the disk.}
\end{figure} 
If we identify the regular fibre with the standard torus in such a way that the homology base of cycles $a,b$ corresponds to the canonical circles in the standard torus, with $b$ the boundary vanishing cycle, the Lefschetz thimbles associated to the 9 Lefschetz singularities with cycle $b$ are again 2-spheres, just like in the previous example. As described in \cite{Cavalcanti2009}, these spheres can be blown down. 

Using this surgery, the Lefschetz thimbles associated to the remaining three Lefschetz singularities are branes with boundary. 
\end{ex} 

\section{Lagrangian branes with boundary and complex branes in holomorphic log symplectic manifolds} \label{sec:hol}
As previously established, Lagrangian branes with boundary as studied in this paper are \emph{not} generalized complex branes in the usual sense. 
We will now consider a stable generalized complex structure on a complex surface $X$ which is given by a holomorphic log Poisson structure $\pi$, or equivalently a holomorphic complex log symplectic form $\Omega=B+i\omega$. 

Lagrangian branes with boundary are never complex submanifolds of such a stable generalized complex manifold, however in this section we present two examples where they are isotopic to complex curves outside the degeneracy locus. In both of these case, the resulting complex curve is non-algebraic, so not part of the commonly studied class of complex submanifolds. While the particular construction here is ad hoc in each example, Lagrangian branes with boundary may be a source of non-algebraic complex curves also in other examples. 

\subsection{Example: $\Cnum\times T^2$} 
Consider $M=\Cnum\times T^2$ with complex coordinates $(w,z)$ and stable GC structure given by 
\[ \Omega = B + i \omega = \frac{\dx w}{w}\wedge \dx z = \frac{\dx r}{r}\wedge \dx x - \dx \theta \wedge \dx y + i \left(\frac{\dx r}{r}\wedge \dx y + \dx \theta \wedge \dx x \right) \] 
Write $w=re^{i\theta}, z=x+iy, (w,z)=(r,\theta,x,y)$. Consider a submanifold $L$ given as follows: 
\[L=\{(r,\theta,\theta,f(r))\},\ f(r) \text{ smooth. } \] 
All such submanifolds are Lagrangian branes: 
\[ \iota^*_L \omega = \frac{\dx r}{r} \wedge \party{f}{r} \dx r + \dx \theta \wedge \dx \theta = 0 \] 

\begin{prop} 
If $L_t = \{(r,\theta,\theta,f_t(r))\} $ is a smooth family of such branes outside $D=\{r=0\}$, we can find a time-dependent Hamiltonian vector field $X_t=\omega^{-1}(\dx g_t)$ with $g_t$ a smooth family of smooth maps whose flow $\phi_t$ reproduces the family $L_t$: 
\[ \phi_t(L_0)=L_t \] 
\end{prop} 

\begin{proof} 
Since $f_t=f_t(r)$, $-\frac{\dx f_t}{\dx t} \frac{\dx r}{r}$ is an exact one-form away from $r=0$, so $-\frac{\dx f_t}{\dx t} \frac{\dx r}{r}=\dx g_t$. 
Now we consider the time-dependent Hamiltonian vector field 
\[ X_t=\omega^{-1}(\dx g_t) = \frac{\dx f_t}{\dx t} \party{}{y} \] 
Its flow satisfies: 
\[ \frac{\dx \phi_t^y}{\dx t} = \frac{\dx f_t}{\dx t},  \] 
so $\phi_t(r,\theta,x,y)=(r,\theta,x,y+f_t(r)-f_0(r))$. \\
Thus $\phi_t(L_0)= \{ \phi_t(r,\theta,\theta,f_0(r)) \} = \{ (r,\theta,\theta,f_t(r)) \} = L_t$.
Note that this flow is everywhere well-defined for all $t\in [0,1]$ and $r>0$. 
\end{proof} 

$L=\{z=-i\log w\}=\{(r,\theta,\theta,-\log r)\}$ defines a cylindrical complex brane in $M$ which does not intersect $D=\{w=0\}$, instead it wraps around the $y$-direction faster and faster as $r\rightarrow 0$. 

Consider the family of branes $ L_t = \{(r,\theta,\theta, (t-1)\log(r+t))\}$.
For $t>0$ this is a family of Lagrangian branes with boundary, and we have $L_0=L, L_1=\{(r,\theta,\theta,0)\}$. 
\[ \phi_t(r,\theta,x,y)=(r,\theta,x,y+(t-1)\log(r+t) +\log(r)) \] 
is the Hamiltonian flow that maps these branes into each other. It is well-defined and smooth away from the anticanonical divisor. The closer one approaches the anticanonical divisor, the more the flow has to move the brane with boundary $\{(r,\theta,\theta,0)\}$ in order to make it complex (or vice versa). 

\subsection{Example: Hopf Surface} 
This example follows exactly the same pattern as the first: We consider the Hopf surface $X$ with the same coordinates and stable generalized complex structure as in Section \ref{sec:Hopf} and show that outside the anticanonical divisor the branes with boundary in Example \ref{ex:Hopf} can be deformed into complex submanifolds: 

\begin{ex}\begin{enumerate}[label=(\roman*)]
\item A complex submanifold of $X$ is given by $L=\{z_1=\text{const.}=a e^{i\sigma}\}$. Now, this obviously intersects the anticanonical divisor at $\{z_0=0\}$ in a point, but does not intersect $\{z_1=0\}$, instead wrapping infinitely often around the $\eta$-direction as $t=\frac{r_0}{a} \rightarrow \infty$. 
$L$ lies over the path 
\[ r_0 \mapsto \left(r_0^2, \frac{1}{2}\log(r_0^2 + a^2)\right) \] 
in the base. In terms of the coordinates $(t,\eta,\theta_0,\theta_1)$: 
\[ L=\left\{\left(t, \log(a) + \frac{1}{2}\log\left(t^2+1\right)\right)\right\}, \]
or in terms of $t'=1/t$: 
 \[ L=\left\{\left(t', \log(a) + \frac{1}{2}\log\left(\frac{1}{t'^2}+1\right)\right)\right\}. \]

We can interpolate between $L=L_0$ and the brane with boundary 
\[ L_1=\{(t,0,\theta,\sigma)\} \] 
with the family of branes 
\[ L_s=\left\{\left(t,f_s(r), \theta,\sigma\right)\right\},\] 
\[ f_s(t)=(1-s)\left(\log(a)+\frac{1}{2}\log\left( \frac{t^2}{(1+st)^2} +1 \right)\right), \]
and this interpolation can again be realised in terms of a time-dependent Hamiltonian vector field, namely 
\[ \party{f_s(t)}{s} \party{}{\eta} \] 

\item Next, consider the following cylindrical Lagrangian, which is parametrised by one complex coordinate $z=r e^{i\theta}, r\neq 0$: 
\[ L:= \left\{ \left(2r^2, \frac{1}{2} \log \left(r^2+ \frac{1}{2r^2}\right), \theta, -\theta\right)\right\} \] 
Under the map $f$, $L$ projects to the path in $[0,1]\times S^1$ 
\[ r \mapsto \left(4 r^4, \frac{1}{2} \log \left(r^2+ \frac{1}{2r^2}\right)\right) \] 
In terms of the complex coordinates $(z_1,z_2)$, this is $L=\left\{ \left(z, \frac{1}{2z}\right)\right\}$. Clearly, this is a complex submanifold that does not intersect the anticanonical divisor $\{z_0=0\}\cup \{z_1=0\}$. 

Now consider the following family of Lagrangians: 
\[ L_s:=\left\{\left(t, f_s(t), \theta,-\theta\right)\right\}, f_s(r)= \frac{1}{2}(1-s)\left(\log(1/2) + \log \left( \frac{t}{1+ts}+\frac{1}{t+s}\right)\right) \] 
This family of Lagrangians interpolates between $L_0=L$ and the brane with boundary (two $S^1$ boundary components, one in each of the two connected components of the anticanonical divisor) 
\[ L_1=\{(t,0,\theta,-\theta)\} \] 
Whenever $s>0$, $L_s$ extends into the the anticanonical divisor at either end. The closer $s$ is to zero, the stronger the brane $L_s$ and its corresponding base path $(t^2, f_s(t))$ wrap in the $\eta$-direction. 

The Hamiltonian vector field which flows $L_0$ into $L_s$ is given by 
$\party{f_s}{s} \party{}{\eta}$. 
\end{enumerate}\end{ex} 

\section{Conclusions and Outlook}
In this text we have studied stable generalized complex manifolds through the lens of their associated elliptic symplectic form, a special example of a \emph{Lie algebroid symplectic form}. We have shown that stable generalized complex structures are related to certain logarithmic symplectic structures via the real oriented blow-up of the anticanonical divisor. 

Stable generalized complex manifolds are in many ways the simplest class of examples of generalized complex manifolds that are neither symplectic nor complex (and include underlying manifolds that do not admit either a symplectic or a complex structure), and since the structure can be described in terms of an elliptic symplectic form, the quest to extend more techniques from symplectic geometry to stable generalized complex geometry is a natural continuation of the work presented here so far. 

The main focus of this paper has been on Lagrangian branes with boundary, a new class of submanifold with boundary for stable generalized complex manifolds that is not included in the previously studied class of generalized complex branes. And yet, since they are Lagrangian with respect to the elliptic symplectic form, and in particular the restricted symplectic form away from the anticanonical divisor, and appear as Lefschetz thimbles in stable generalized complex Lefschetz fibrations, they can be expected to appear in the construction of a generalisation of the Fukaya category for stable generalized complex manifolds, either considering non-compact Lagrangians in the symplectic bulk, or in a possible adaptation of the Fukaya-Lefschetz approach. 

\bibliography{biblio}
\bibliographystyle{amsalpha} 

\end{document}